\documentclass[reqno]{amsart}
\usepackage{amsmath,amsthm,amssymb}
\usepackage{latexsym}
\usepackage{eucal}
\usepackage{fullpage}

\usepackage{hyperref}
\usepackage{mathabx}

\newtheorem{theorem}{Theorem}[section]
\newtheorem{lemma}{Lemma}[section]
\newtheorem{corollary}{Corollary}[section]

\def\cal{\mathcal}
\let\Re=\undefined
\DeclareMathOperator{\Re}{Re}
\let\Im=\undefined
\DeclareMathOperator{\Im}{Im}

\def\ge{\geqslant}\def\le{\leqslant}
\def\~{\widetilde}

\numberwithin{equation}{section}

\begin{document}

\Large
\title[The growth of polynomials  orthogonal on the unit  circle with respect to a weight $w$ that satisfies $w,w^{-1}\in L^\infty(\mathbb{T}$
 ]{The growth of polynomials orthogonal on the unit circle with respect to a weight $w$ that satisfies $w,w^{-1}\in L^\infty(\mathbb{T})$ }
\author{ S. Denisov }
\address{
\begin{flushleft}
Sergey Denisov: denissov@wisc.edu\\\vspace{0.1cm}
University of Wisconsin--Madison\\  Mathematics Department\\
480 Lincoln Dr., Madison, WI, 53706,
USA\vspace{0.1cm}\\and\\\vspace{0.1cm}
Keldysh Institute for Applied Mathematics, Russian Academy of Sciences\\
Miusskaya pl. 4, 125047 Moscow, RUSSIA\\
\end{flushleft}
}\maketitle

\begin{abstract}
We consider the polynomials $\{\phi_n(z,w)\}$ orthogonal on the circle with respect to a weight $w$ that satisfies $w,w^{-1}\in L^\infty(\mathbb{T})$ and show that $\|\phi_n(e^{i\theta},w)\|_{L^\infty(\mathbb{T})}$ can grow in $n$ at a certain rate.

\end{abstract}\vspace{1cm}
\section{Introduction}

Consider a finite positive measure $\sigma$ defined on the unit
circle and assume that it has infinitely many growth points. Let the polynomials $%
\{\phi_k(z,\sigma)\}$ be orthonormal with respect to measure
$d\sigma$, i.e.,
\begin{equation}\label{1}
\int_{-\pi}^{\pi}{\phi_n(e^{i\theta})}\,\overline{\phi_m(e^{i\theta})}\,d\sigma\,=\delta_{n,m}\;,\quad
n,m=0,1,2\,\ldots, \quad {\rm coeff}(\phi_k,k)>0\,,
\end{equation}
where ${\rm coeff}(Q,k)$ denotes the coefficient in front of $z^k$
in the polynomial $Q$.   Besides the orthonormal polynomials, we can
define the monic orthogonal ones $\{\Phi_n(z,\sigma)\}$ by requiring
\[
\deg \Phi_n=n, \,{\rm
coeff}(\Phi_n,n)=1,\,\int_{-\pi}^{\pi}\Phi_n(e^{i\theta},\sigma)\overline{\Phi_m(e^{i\theta},\sigma)}\,d\sigma\,=0\;,\quad
m<n\,.
\]
Later, we will need to use the following notation. Let $\mathbb{P}_n$ denote the space of polynomials of degree at most $n$, i.e., 
\[
Q\in \mathbb{P}_n \,\Leftrightarrow \,\, Q(z)=q_nz^n+\ldots+q_1z+q_0, \,\, q_j\in \mathbb{C}, j=0,\ldots,n.
\]
For each $n\in \mathbb{Z}^+$, we introduce the operation $(\ast)$ defined on $\mathbb{P}_n$ as follows (see, e.g., \cite{sim1}, formula (1.1.6)):
\[
Q(z)\stackrel{(\ast)}{\longrightarrow} Q^*(z):=\bar{q}_0
z^n+\ldots+ \bar{q}_n\,.
\]
This $(\ast)$ depends on $n$. Since $Q^*(z)=z^n\overline{Q(\overline{z}^{-1})}$, we can make a trivial observation that 
\begin{equation}\label{form12}
Q^*(z)=z^n\overline{Q(z)}, \,|Q(z)|=|Q^*(z)|\,,
\end{equation}
if $z\in \mathbb{T}$. Moreover, $(Q^*)^*=Q$.

One of the central questions in the theory of orthogonal polynomials is
to understand the behavior of $\{\phi_n(z,\sigma)\}$ as
$n\to\infty,z\in \mathbb{T}$ given some regularity assumption on
$\sigma$.  Consider the following class of measure.

 {\it
Given $\delta\in (0,1)$, we say that $\sigma \in S_\delta$ if
$\sigma$ is a probability measure and
\begin{equation}
\sigma'\ge \delta/(2\pi), \quad {\rm a.e. \,\,\theta\in [-\pi,\pi]}\,.
\end{equation}
} We will call $S_\delta$ the Steklov class. One version of the
celebrated Steklov conjecture \cite{1} is to decide whether the sequence
$\{\phi_n(z,\sigma)\}$ is bounded in $n$ for every $z\in \mathbb{T}$
provided that $\sigma\in S_\delta$. This conjecture was solved
negatively by Rakhmanov in 1979 \cite{rakh1} (see
\cite{5,Ger1,Ger2,Ger3,Gol,2}). This development lead to more
general problem (dubbed ``problem by Steklov") which asks to obtain
the bounds on $\|\phi_n(z,\sigma)\|_{L^\infty(\mathbb{T})}$ that
would be sharp when $n\to\infty$ and $\sigma\in S_\delta$. Some
upper estimates are usually  not very difficult to come by. For
example, one has

\begin{lemma}
If $\sigma\in S_\delta$, then
\begin{equation}\label{steklo2}
\|\phi_n(z,\sigma)\|_{L^\infty(\mathbb{T})}=o(\sqrt n), \quad
n\to\infty\,.
\end{equation}
\end{lemma}
This result follows from, e.g., \cite{mnt}, Theorem 4 (see also
\cite{nt},\cite{rw} for the real-line case).

If $n$ is fixed, then the following Lemma is true.

\begin{lemma}{\rm (\cite{adt})}\,
Define $M_{n,\delta}$ by
\[
M_{n,\delta}:=\sup_{\sigma\in S_\delta}
\|\phi_n(z,\sigma)\|_{L^\infty(\mathbb{T})}\,.
\]
Then, we have
\[
M_{n,\delta}\le \min \left\{ \sqrt{\frac{n+1}{\delta}},
\frac{1}{\sqrt\delta}\left(1+\sqrt{\frac{n(1-\delta)}{\delta}}\right)
\right\}\,.
\]
\end{lemma}
The question then naturally arises whether these results are optimal
when $\delta$ is fixed and $n\to\infty$. In his second
paper~\cite{rakh2}, Rakhmanov obtained the estimates that were
nearly sharp. However, the full solution to the problem has been
given quite recently and by a different method.
\begin{theorem}{\rm (\cite{adt})}\label{T3-i}         If $\delta\in (0,1)$ is fixed, then
\begin{equation}\label{osnova-i}M_{n,\delta} > C_{(\delta)}\sqrt n\,\, .
\end{equation}
\end{theorem}
\smallskip
\begin{theorem}{\rm (\cite{adt})}\label{rrra-i}            Let $\delta{\in}(0,1)$ be fixed. Then, for every positive sequence
$\{\beta_n\}$ that satisfies $\lim_{n\to\infty}\beta_n=0$, there is an absolutely continuous probability
measure $\sigma^*:d\sigma^*={\sigma^*}'d\theta
$ such that $ \sigma^*{\in}S_\delta$ and
\begin{equation}
\label{est-ra-i}
\|\phi_{k_n}(z,\sigma^*)\|_{L^\infty(\mathbb T)}\ge
\beta_{k_n}\sqrt{k_n}
\end{equation}
for some sequence $\{k_n\}\subset\mathbb N$.\end{theorem} We need to
make three remarks here.\smallskip

{\bf Remark 1.}
The paper \cite{adt} proves more than just the inequality
\eqref{osnova-i}. It shows that for every $p\in
[1,\infty)$ and every $L_1>1$, the following bound holds
\begin{equation}\label{break}
\sup_{w\ge
\delta/(2\pi),\|w\|_{L^1(\mathbb{T})}=1,\|w\|_{L^p(\mathbb{T})}<L_1}\|\phi_n(z,w)\|_{L^\infty(\mathbb{T})}>C_{(\delta,p,L_1)}\sqrt
n\,.
\end{equation}
Here $\{\phi_n(z,w)\}$ are orthonormal with respect to the weight $w$.  In general, we call  function $w$  a weight if $w\in L^1(-\pi,\pi)$, $w\ge 0$, and $w\nequiv 0$.  \smallskip

 {\bf Remark 2.} We will say that the measure $\sigma$ belongs to the Szeg\H{o} class, if 
\begin{equation}\label{szego5}
\int_{-\pi}^{\pi} \log \sigma' d\theta>-\infty.
\end{equation}
It is known (\cite{sim1}, formulas (1.5.78) and (2.3.1)), that
 \begin{equation}\label{fact11}
\exp\left(\frac{1}{4\pi}\int_{\mathbb{T}} \log(2\pi \sigma'(\theta))
d\theta\right)\le 
\left|\frac{\Phi_n(z,\sigma)}{\phi_n(z,\sigma)}\right|\le 1, \quad
\forall z\in \mathbb{C}\,.
\end{equation}
and the left hand side is positive if and only if \eqref{szego5} holds.
Thus, for measures in Steklov class,  we have a bound
\begin{equation}\label{fact12}
\sqrt\delta\le 
\left|\frac{\Phi_n(z,\sigma)}{\phi_n(z,\sigma)}\right|\le 1, \quad
\forall z\in \mathbb{C}\,.
\end{equation}
Therefore, if $\{\Phi_n\}$ or $\{\phi_n\}$ grows in $n$, then the other sequence grows as well.\smallskip

{\bf Remark 3.} The requirement that the measure is a probability
one is not so restrictive due to the following scaling
\begin{equation}\label{scale11}
\phi_n(z,\sigma)=\alpha^{1/2}\phi_n(z,\alpha\sigma), \quad
\Phi_n(z,\sigma)=\Phi_n(z,\alpha\sigma), \quad \alpha>0\,.
\end{equation}\smallskip

Suppose we are given a function $h(t)$ defined on $[0,2\pi]$ which satisfies
\[h(0)=\lim_{t\to 0}h(t)=0.
\]
We will say that
a function $w$, defined on $\mathbb{R}$ and $2\pi$-periodic there, has $h$ as its modulus of continuity
with a constant $L\in [0,\infty)$,  if the estimate
\begin{equation}\label{moduc}
|w(x_2)-w(x_1)|\le L h(|x_2-x_1|)
\end{equation}
holds for all $x_1,x_2\in \mathbb{R}$ that satisfy $|x_1-x_2|\le 2\pi$.

If $w$ is defined on $[-\pi,\pi]$ and $w\in C[-\pi,\pi]$, $w(-\pi)=w(\pi)$,  we can extend it to all of $\mathbb{R}$ as $2\pi$-periodic function. Then, if \eqref{moduc} holds, we will write $w\in C_{h,L}(\mathbb{T})$.

\begin{theorem}{\rm  ((\cite{szego}),} Theorem 12.1.3) \label{bernstein} Assume that
$d\sigma=w(\theta)d\theta$ and $w\in C_{h,L}(\mathbb{T})$,
$h(t)=|\log(t/8)|^{-1-\epsilon}$ with some $\epsilon>0$ and $L$.
Then, as long as $w\ge\delta>0$ for all $\theta\in [-\pi,\pi]$, we
have
\begin{equation}
\phi_n(e^{i\theta})=e^{in\theta} \overline{\Pi(e^{i\theta})}+\epsilon_n(e^{i\theta}), \quad n\to\infty\,,
\end{equation}
where $\epsilon_n(e^{i\theta})\to 0$ uniformly over $\theta\in [-\pi,\pi]$. The function $\Pi(z)$ is defined as the outer function in $\mathbb{D}$
that satisfies $|\Pi(e^{i\theta})|^{-2}=2\pi w(\theta),  \, \Pi(0)>0$.
\end{theorem}

In \cite{murman}, given arbitrary $\epsilon>0$ and $L>0$, Ambroladze
constructed a positive weight $\widehat w\in C_{h,L}(\mathbb{T})$ ,
$h(t)=|\log(t/8)|^{-1+\epsilon}$ for which $\{\|\phi_n(z,\widehat
w)\|_{L^\infty(\mathbb{T})}\}$ is unbounded in $n$. This showed sharpness of regularity assumption in Theorem \ref{bernstein}.\smallskip

The bound \eqref{break}  raises a question: what regularity of a weight can
improve the $\sqrt n$ upper bound? In this paper, we give partial
answer to this question. We  consider measures given by the weights $w$ that satisfy  $w,w^{-1}\in L^\infty(\mathbb{T})$ and obtain upper and lower estimates for the possible growth of the supremum norms. In the second section, we provide an argument due to Fedor Nazarov that gives the upper bounds for $\|\phi_n(e^{i\theta},w)\|_{L^\infty(\mathbb{T})}$. The third section contains the proofs of the lower bounds. The Appendices have some auxiliary statements that are used in the main text.\smallskip

We will use the following definitions and  notation. 

\begin{itemize}
\item The Schwartz kernel $C(z,\xi)$ is
defined as
\[
C(z,\xi):=\frac{\xi+z}{\xi-z},\quad\xi\in\mathbb T,\,z\in \mathbb{C}\,.
\]
\item 
The function analytic in $\mathbb{D}=\{z:|z|<1\}$ is called
Caratheodory function if its real part is nonnegative
in~$\mathbb{D}$.

\item Given a set $\Omega$, $\chi_\Omega$ denotes the
characteristic function of $\Omega$.

\item If two positive functions
$f_{1}$ and $f_2$ are given, we write $f_1\lesssim f_2$ if there is an
absolute constant $C$ such that
\[
f_1<Cf_2
\]
for all values of the argument. We define $f_1\gtrsim f_2$
similarly. Writing $f_1\sim f_2$  means $f_1\lesssim f_2\lesssim
f_1$. 

\item The symbol $P_{[n_1,n_2]}$ denotes the $L^2[-\pi,\pi]$ orthogonal
Fourier projection to the Fourier modes $\{n_1,\ldots,n_2\}$:
\[
P_{[n_1,n_2]}f:=\sum_{j=n_1}^{n_2}(2\pi)^{-1/2}e^{ij\theta}\widehat f_j
\]
and 
\[
\widehat f_j:=(2\pi)^{-1/2}\int_{-\pi}^\pi f(\theta)e^{-ij\theta}d\theta
\]
stands for the Fourier coefficients of $f\in L^1(-\pi,\pi)$.
\item The symbol $\cal{F}_n$ will denote the Fejer kernel of degree $n$ given by
\[
\cal{F}_n:=\frac{1}{2\pi n}\left(  \frac{\sin(n\theta/2)}{\sin(\theta/2)}        \right)^2\,.
\]

\item
The Jackson kernel is given by
\[
\cal{K}_n=c_n\cal{F}_n^2\sim n^{-3}\frac{\sin^4(nx/2)}{\sin^4(x/2)}
\]
and $c_n$ is chosen such that
\[
\int_{-\pi}^\pi \cal{K}_ndx=1\,.
\]

\item Throughout the paper, the symbol $C$ denotes a positive absolute constant. Its actual value might change from formula to formula. If, within a certain proof or statement of the result, we write $C_1, C_2$, etc., this means that we want to emphasize that these constants might be different. Sometimes we will use $c$ instead of $C$ to emphasize the difference.

\item If the proof uses parameters, e.g., $\alpha_1,\ldots, \alpha_n$, we  write $C_{(\alpha_1,\ldots,\alpha_n)}$ or $n_{(\alpha_1)}$ to denote a non-negative function defined for all values of these parameters. We use this notation  in those cases when the actual dependence on these parameters is not important (or unknown) to us. Similarly to the constants, if we want to distinguish between them to avoid confusion, we  write, e.g., $C^{(1)}_{(\alpha_1)}, C^{(2)}_{(\alpha_2)}$, etc.

\item If $(\Omega_{1(2)},\mu_{1(2)})$ are two measure spaces and  $A$ is a linear operator, bounded from $L^{p_1}(\Omega_1,\mu_1)$ to $L^{p_2}(\Omega_2,\mu_2)$, then  $\|A\|_{p_1,p_2}$ denotes its operator norm.

\item The following standard notation will be used several times. If $f_1$ is non-negative function, then writing $O(f_1)$ denotes a function, say, call it $f$ for a moment, which satisfies 
\[
|f|\lesssim f_1
\] 
for the specified range of the arguments. For example, writing 
\[
O(\epsilon |\theta|),\quad  \epsilon\in (0,1), \theta\in [-\pi,\pi)
\] denotes a function, say $f$, which satisfies
\[
|f(\theta,\epsilon)|\lesssim \epsilon |\theta|
\]
for $\theta\in [-\pi,\pi),\epsilon\in (0,1)$. In that case, $f_1=\epsilon |\theta|$.
\item If the function $f$ is defined on $\mathbb{T}$, the symbol $\|f\|_p$ means $\|f\|_{L^p(\mathbb{T})}, \,1\le p\le\infty$. 

\item Given two functions $f_1,f_2\in L^1(\mathbb{T})$, we define their convolution by
\[
f_1\ast f_2:=\int_{\mathbb{T}} f_1(\theta-\phi)f_2(\phi)d\phi\,.
\]
\item If $z\in \mathbb{C},z\neq 0$, its argument $\arg z$ is always chosen such that it belongs to $(-\pi,\pi]$.

\end{itemize}
\bigskip

\section{Upper bounds a la Bernstein}

In this section, we will apply an idea that was used by S.N.~Bernstein when proving an analog of Theorem~\ref{bernstein} for the
case of orthogonality on the segment of real line \cite{snb}. In fact, Szeg\H{o}'s proof of Theorem \ref{bernstein} is an adaptation of this method.

\begin{lemma}The monic polynomial $Q_n$ of degree $n$  is $n$-th monic orthogonal polynomial with respect
to a weight $w$ if and only if
\begin{equation}\label{cut}
{P}_{[0,n-1]}(wQ_n)=0\,.
\end{equation}
\end{lemma}
\begin{proof}It is sufficient to notice that \eqref{cut} is
equivalent to
\[
\int_{-\pi}^{\pi} Q_ne^{-ij\theta}wd\theta=0, \quad j=0,\ldots,n-1\,,
\]
which is the orthogonality condition.
\end{proof}

\begin{lemma}\label{hilb}
For every $p\in [2,\infty)$,
\begin{equation}\label{karas}
\|{P}_{[0,n]}\|_{p,p}\le 1+C(p-2)\,.
\end{equation}
\end{lemma}
\begin{proof}If $\cal{P}^+$ is the projection of $L^2(\mathbb{T})$ onto $H^2(\mathbb{T})$ (Riesz projection), then we can write an identity
\[
{P}_{[0,n]}=\cal{P}^+-z^{n+1} \cal{P}^+ z^{-(n+1)}\,.
\]
In \cite{iv}, it was proved that
\[
\|\cal{P}^+\|_{p,p}=\frac{1}{\sin(\pi/p)}, \quad 1<p<\infty\,.
\]
Thus \[\|P_{[0,n]}\|_{p,p}\le \frac{2}{\sin(\pi/p)}\]
by triangle inequality. On the other hand, we have $\|P_{[0,n]}\|_{2,2}=1$ and the Riesz-Thorin Theorem \cite{folland} allows to interpolate between $p=2$ and, e.g., $p=3$ to get a bound
\[
\|P_{[0,n]}\|_{p,p}\le 1+C(p-2), \quad 2<p<3\,.
\]
Since
\[
\frac{2}{\sin(\pi/p)}\sim p
\]
for $p\ge 3$, the proof is finished.
\end{proof}

The proof of the following result is due to Fedor Nazarov (personal communication). We give it here for completeness.

\begin{theorem}{\it (F. Nazarov)}\label{fodd}
If $\epsilon\in (0,1]$, then
\begin{equation}\label{contract1}
\sup_{1\le w\le 1+\epsilon} \|\Phi_n(e^{i\theta},w)\|_p\lesssim 1, \quad p=C_1\epsilon^{-1}\,.
\end{equation}
If $T>2$, then
\begin{equation}\label{contract2}
\sup_{1\le w\le T} \|\Phi_n(e^{i\theta},w)\|_p \lesssim 1, \quad p=2+C_2T^{-1}\,.
\end{equation}
\end{theorem}
\begin{proof}
Let $\mu=w-1$. Following Bernstein, we use \eqref{cut} to get the formula
\begin{eqnarray*}
\Phi_n(z,w)=z^n+(2\pi)^{-1}\sum_{j=0}^{n-1} z^j \int_{\mathbb{T}}
\Phi_n(e^{i\theta},w)e^{-ij\theta}d\theta\\=z^n-(2\pi)^{-1}\sum_{j=0}^{n-1}
z^j \int_{\mathbb{T}}
\Phi_n(e^{i\theta},w)e^{-ij\theta}\mu(\theta)d\theta
=z^n-{P}_{[0,n-1]}(\mu \Phi_n)\,.
\end{eqnarray*}
Thus, $\Phi_n$ is a fixed point of the operator defined by
\[
f\to z^n-{P}_{[0,n-1]}(\mu f)\
\]
on $f\in L^p(\mathbb{T})$. If $1\le w\le 1+\epsilon$, then $\|P_{[0,n-1]}\mu\|_{p,p}\le Cp\epsilon$ by \eqref{karas}. Choosing $p$ such that $Cp\epsilon<1$ makes $P_{[0,n-1]}\mu$ a contraction in $L^p(\mathbb{T})$. Then, \eqref{contract1} follows from the Fixed Point Theorem for contractions.

To prove \eqref{contract2}, we need to modify this argument a little. From \eqref{cut}, we get
\begin{equation}\label{goreg}
\Phi_n=z^n+P_{[0,n-1]}(1-\kappa w)\Phi_n\,,
\end{equation}
where $\kappa$ is an arbitrary real parameter. Taking $\kappa=T^{-1}$, we can write
\[
0\le 1-\kappa w\le 1-T^{-1}
\]
and \eqref{karas} implies
\[
\|P_{[0,n-1]}(1-\kappa w)\|_{p,p}\le (1+C(p-2))(1-T^{-1})=1-(1-C \lambda)T^{-1}-C\lambda T^{-2}\,,
\]
provided that $p=2+\lambda T^{-1}$. Choosing $\lambda<C^{-1}$  makes the right hand side of \eqref{goreg} a contraction in $L^p(\mathbb{T})$. This again proves \eqref{contract2} by the Fixed Point Theorem if we let \mbox{$C_2:=\lambda$.}
\end{proof}

By Nikolskii inequality for elements of $\mathbb{P}_n$, we have
\begin{corollary}
If $\epsilon\in (0,1]$, then
\begin{equation}\label{contract11}
\sup_{1\le w\le 1+\epsilon} \|\Phi_n(e^{i\theta},w)\|_\infty\lesssim  n^{C_3\epsilon}\, .
\end{equation}
If $T>2$, then
\begin{equation}\label{contract12}
\sup_{1\le w\le T} \|\Phi_n(e^{i\theta},w)\|_\infty\lesssim n^{\frac 12-C_4T^{-1}}\,.
\end{equation}
\end{corollary}
\begin{proof} We have an estimate (see, e.g., \cite{zygmund}, p.154, Theorem 6.35)
\[
\|\Phi_n\|_\infty\lesssim n^{1/p}\|\Phi\|_p, \quad p\ge 2\,.
\]
Taking $p$ as in \eqref{contract1} and \eqref{contract2} finishes the proof. 
\end{proof}

{\bf Remark.} The estimate \eqref{contract12} shows that \eqref{break} is no longer true if $w$ is in Steklov class and belongs to $L^\infty(\mathbb{T})$.

{\bf Remark.} Due to \eqref{fact11} and \eqref{scale11}, the analogous results for orthonormal polynomials and for probability measures  hold as well.\bigskip

\section{Lower bounds}

The following Theorem is the main result of the paper.

\begin{theorem}\label{mainn}
If $\epsilon\in (0,1]$, then
\begin{equation}\label{contra1}
\sup_{1\le w\le 1+\epsilon} \|\Phi_n(e^{i\theta},w)\|_\infty>C^{(1)}_{(\epsilon)} n^{C_5\epsilon}\,.
\end{equation}

If $T>2$, then
\begin{equation}\label{contra2}
\sup_{1\le w\le T} \|\Phi_n(e^{i\theta},w)\|_\infty>C^{(2)}_{(T)}n^{\frac 12-C_6T^{-1/4}}\,.
\end{equation}

\end{theorem}

{\bf Remark.} The constants $C$ in these bounds are different from those in the corollary above. 
The bound \eqref{contra1} implies that for every $\alpha_1\in (0,C_5)$, there is $n_{(\epsilon,\alpha_1)}\in \mathbb{N}$  such that 
\[
\sup_{1\le w\le 1+\epsilon} \|\Phi_n(e^{i\theta},w)\|_\infty> n^{\alpha_1\epsilon}\,
\]
for all $n>n_{(\epsilon,\alpha_1)}$. Analogous claim can be made for 
\eqref{contra2}.

 The main tool for proving this Theorem will be
 the technique developed in \cite{adt} to obtain the sharp estimates in Steklov's problem (\cite{adt},
\cite{den-jat}, see Appendix~A), along with the localization
principle (\cite{adt}, Appendix~B).

\begin{proof}{\it (of Theorem \ref{mainn})}. The proof will consist of two parts (estimate \eqref{contra1} and estimate \eqref{contra2}).  In the first one,
we will construct a weight having deviation from a constant at
most $C\epsilon$ on a small arc around the point $\theta=0$. The
corresponding orthogonal polynomial will be large at $z=1$. Then,
using the localization principle, which is discussed in Appendix~B, we will provide a weight of small
deviation over the whole circle and show that the corresponding polynomial still has the required size.
In the second part, we will apply a similar strategy to handle the weights with large deviation.

\subsection{The case of small deviation, estimate \eqref{contra1}}
We will follow the strategy used in \cite{adt} and explained in Lemma~\ref{decop} in Appendix A.
Let us assume that $n$ is even, the case of odd $n$ can be handled similarly. 
Thus, we will be applying this Lemma taking $2n$ instead of $n$.  Below, we introduce $\widetilde F$, $\phi_{2n}^*$ and
check that they satisfy conditions of Lemma~\ref{decop}. Then, we
will control how the weight $\sigma'$ in \eqref{mp} deviates from a constant. 
Notice that, when proving \eqref{contra1}, it is sufficient to consider $\epsilon\in (0,\epsilon_0)$ where $\epsilon_0\in (0,1)$ is fixed. We can also assume that $n>n_{(\epsilon)}$.    \smallskip

{\bf (a)} Consider auxiliary function $h_n$ given by
\[
h_n:=2(1-e^{i\theta})^{\epsilon}\ast \cal{F}_n\,.
\]
Clearly $h_n$ is a
polynomial of degree $n-1$. Since $\Re (1-e^{i\theta})^\epsilon\geq 0$ and $\cal{F}_n\geq 0$,  its real part is strictly
positive over $\mathbb{T}$ so $h_n$ is zero-free in
$\overline{\mathbb{D}}$. In Lemmas \ref{frukt} and \ref{lemma1} of Appendix~C, more
detailed information is obtained, in particular,
\begin{equation}\label{arga}
|{\arg}\, h_n|\lesssim \epsilon
\end{equation}
uniformly in $\theta\in [-\pi,\pi)$ and $n>n_{(\epsilon)}$.

Because $\int_{-\pi}^\pi \cal{F}_nd\theta=1$, we get
\begin{equation}\label{form13}
\int_{-\pi}^\pi h_n d\theta=2\pi h_n(0)=4\pi\,.
\end{equation}
Since
$\cal{F}_n$ is even and \mbox{$\Im (1-e^{i\theta})^{\epsilon}$} is
odd on $[-\pi,\pi]$, we know that $\Im H_n$ is odd on $[-\pi,\pi]$ as well.\smallskip

{\bf  (b)} In Lemma \ref{lemma1}, choose 
\begin{equation}\widetilde F:=\frac{2}{h_n}\,.\label{defan}
\end{equation}
From the properties of $h_n$,
we get analyticity of $\widetilde F_n$ in $\mathbb{D}$ and infinite
smoothness on $\mathbb{T}$. Since
\begin{equation}\label{inva}
\Re \widetilde F=2\frac{\Re h_n}{|h_n|^2},
\end{equation}
$\widetilde F$ is Caratheodory function. Moreover, since
$h_n(0)=2$, we get the normalization \mbox{$\Re \widetilde F(0)=1$}.

From \eqref{arga} and \eqref{inva}, we also have

\begin{equation}\label{nata}
1\le\Re h_n \cdot \Re \widetilde F\le 2
\end{equation}
uniformly in $\theta\in [-\pi,\pi), n>n_{(\epsilon)}$.\smallskip

{ \bf  (c)} Notice that $\Bigl( \Re \widetilde F \Bigr)\ast \cal{F}_n$ is a positive trigonometric polynomial of degree at most $n-1$.
Let $q_n\in \mathbb{P}_{n-1}$ be 
zero free in $\mathbb{D}$ and satisfy
\begin{equation}\label{setit}
|q_n|^2=\Bigl( \Re \widetilde F \Bigr)\ast \cal{F}_n, \quad z\in \mathbb{T}, \quad q_n(0)>0\,.
\end{equation}
The existence and uniqueness of this polynomial $Q_n$ is well-known (\cite{szego}, Theorem 1.2.2).
We will need this $Q_n$ in the next argument.

In Lemma \ref{lemma1}, we also need to choose $\phi_{2n}$. We will first specify $\phi_{2n}^*$ and then will take $\phi_{2n}:=(\phi_{2n}^*)^*$. In this calculation, the operation $(\ast)$ acts on $\mathbb{P}_{2n}$. 

We let 
\begin{equation}\label{popug}
\phi^*_{2n}:=\alpha_n(q_n+q_n^*+q_nh_n)\,,
\end{equation} where $\alpha_n$ is
the positive normalization parameter chosen to ensure that
\begin{equation}\label{form15}
\|\phi_{2n}^{-1}\|^2_{L^2(\mathbb{T})}=\|(\phi_{2n}^*)^{-1}\|^2_{L^2(\mathbb{T})}=2\pi\,.
\end{equation}
as required in \eqref{norma}. We will  obtain the estimates for $\alpha_n$ below. 
We emphasize again  that $Q_n^*$ in \eqref{popug} is
understood as $(\ast)$--operation on $\mathbb{P}_{2n}$. Thus,
$\deg \phi_{2n}^*\le 2n$.

To make sure that $\phi^*_{2n}$ has no zeroes in
$\overline{\mathbb{D}}$, we  write
\begin{equation}\label{bez-nuley}
q_n+q_n^*+q_nh_n=q_n\left(1+h_n+\frac{q_n^*}{q_n}\right)\,.
\end{equation}
Then, $q_n$ is zero-free in $\overline{\mathbb{D}}$ and $1+h_n+{q_n^*}/{q_n}$ is analytic in
$\mathbb{D}$ having the positive real part since
\[
\Re h_n>0, \quad \Re \left(1+\frac{q_n^*}{q_n}\right)\ge 0, \quad
z\in \mathbb{T}\,.
\]
The last inequality is the consequence of $|q_n^*|=|q_n|, z\in
\mathbb{T}$ (see \eqref{form12}).\smallskip

At the point $z=0$, we have
\[
q_n(0)+q_n^*(0)+q_n(0)h_n(0)=q_n(0)(1+h_n(0))\,,
\]
because $q_n^*(0)=0$. Then, \eqref{form13} and \eqref{setit} ensure that the right hand side in the last expression is positive. Since the polynomial $q_n+q_n^*+q_nh_n$ is positive at zero, we get $\deg \phi_{2n}=\deg ((q_n+q_n^*+q_nh_n)^*)=2n$ because, again, $(\ast)$ acts on $\mathbb{P}_{2n}$.

Next, we will show that
\begin{equation}\label{updown}
\alpha_n\sim 1\,,
\end{equation}
 if $n>n_{(\epsilon)}$. Indeed, the following estimate holds
\begin{equation}\label{ocenka-sni}
|q_n+q_n^*+q_nh_n|^2=\left|q_n\left(1+h_n+\frac{q_n^*}{q_n}\right)\right|^2\ge
|q_n|^2(\Re h_n)^2=(\Re h_n)^2(\Re \widetilde
F\ast \cal{F}_n)\,,
\end{equation}
where we have used \eqref{setit} at the last step. 
Let us rewrite the last expression using \eqref{inva}
\[
(\Re h_n)^2(\Re \widetilde
F\ast \cal{F}_n)=(\Re h_n)^2\Re \widetilde F\cdot \frac{(\Re \widetilde F\ast
\cal{F}_n)}{\Re \widetilde F}=    2\frac{(\Re h_n)^3}{|h_n|^2} \cdot \frac{(\Re \widetilde F\ast
\cal{F}_n)}{\Re \widetilde F}\,.
\]
For the second factor, we have
\begin{equation}\label{form14}
\left| \frac{(\Re \widetilde F\ast
\cal{F}_n)}{\Re \widetilde F}\right|\sim 1\,,
\end{equation}
if $\theta\in [-\pi,\pi)$ and  $n>n_{(\epsilon)}$. This is due to
Lemma~\ref{lemma2}.

 From \eqref{arga}, we get
\[
\frac{(\Re  h_n)^2 }{| h_n|^2}\sim 1
\]
under the same assumptions as in \eqref{form14}.

Thus, $|q_n+q_n^*+q_nh_n|^2\gtrsim \Re h_n$ if $\theta\in [-\pi,\pi),
n>n_{(\epsilon)}$. Therefore, Lemmas \ref{frukt}, \ref{lemma1}, \ref{lemma2}, and \eqref{nata}   give
\[
|\theta|^\epsilon\lesssim  |q_n+q_n^*+q_nh_n|^2\lesssim |q_n|^2=\Bigl( \Re \widetilde F \Bigr)\ast \cal{F}_n\lesssim |\theta|^{-\epsilon}
\]
for $\theta\in [-\pi,\pi)$, uniformly in $n>n(\epsilon)$.
Therefore, we have \eqref{updown}.\bigskip

Now that we checked all conditions in Lemma \ref{decop}, we can apply it. Consider the value of $\phi_{2n}^*$ at $z=1$. Since $|q_n(e^{i\theta})|$ is even and $q_n(0)>0$, we have $q_n(1)\in \mathbb{R}$. Thus, $q_n^*(1)=q_n(1)$ and
\[
|\phi_{2n}^*(1,\sigma)|^2= \alpha_n^2|q_n(1)(2+h_n(0))|^2\sim |\Re\widetilde
F\ast \cal{F}_n|_{z=1}\ge
\]
\[
 |\Re \widetilde F|_{z=1}-|\Re
\widetilde F|_{z=1}\left|    \frac{\Re \widetilde F\ast \cal
F_n}{\Re \widetilde F}-1\right|_{z=1}\gtrsim |\Re \widetilde F|_{z=1}
\]
by Lemma \ref{lemma2}, provided that $n>n_{(\epsilon)}$.  Now, \eqref{inva} and Lemma \ref{lemma1}
yield
\[
|\phi_{2n}(1,\sigma)|^2=|\phi_{2n}^*(1,\sigma)|^2\gtrsim n^\epsilon\,,
\]
which ensures the necessary growth (compare with \eqref{contra1}).\smallskip

Next, we will study the weight $\sigma'$ given by the formula  \eqref{mp}.
We can write
\[
\sigma'^{-1}=\upsilon_n \cal{ABC}\,,
\]
where $\upsilon_n=\pi\alpha_n^2/2$ is $n$--dependent parameter
and
\begin{equation}\label{brown1}
\cal{A}:=\frac{|q_n|^2}{\Re \widetilde F},\quad
\cal{B}:=|2+\overline{h}_n(1-\widetilde F)|^2, \quad
\cal{C}:=\left|\xi+\frac{2+h_n(1+\widetilde
F)}{2+\overline{h}_n(1-\widetilde F)}\right|^2, \quad
\xi=e^{i(2n\theta-2\Theta)}\,,
\end{equation}
where $\Theta=\arg q_n$. The estimate \eqref{updown} guarantees that $\upsilon_n\sim 1$ if $n>n_{(\epsilon)}$.\smallskip

Now, in what follows, we will consider a small $\epsilon$--dependent
interval $I_\epsilon\subset [-\pi,\pi)$ centered at $\theta=0$ and will control how each of the factors $\cal{A}$
and $(\cal{BC})$ deviates from constants on $I_\epsilon$. The size of $I_\epsilon$ will be specified later, it will not depend on $n$ if $n>n_{(\epsilon)}$.

By Lemma \ref{lemma2} and the choice of $q_n$,
\begin{equation}\label{blue7}
|\cal{A}-1|\lesssim \epsilon
\end{equation}
uniformly in $\theta\in [-\pi,\pi)$ and $n>n(\epsilon)$.

For $\cal{B}$, we can use \eqref{defan} to write
\[
\cal{B}=|2+\overline{h}_n(1-\widetilde F)|^2=\left|2\left(1-\frac{\overline
{h}_n}{h_n}\right)+\overline{h}_n\right|^2\lesssim \epsilon^2+|h_n|^2\sim
\epsilon^2+n^{-\epsilon}+|\theta|^\epsilon
\]
due to estimates on $h_n$ obtained in Lemma
\ref{lemma1}), provided that $n>n_{(\epsilon)}$. 

Now, we choose $I_\epsilon$ such that $\max_{\theta\in I_\epsilon}|\theta|^\epsilon<\epsilon^2$. Then, we have
\[
\cal{B}\lesssim \epsilon^2
\]
for $\theta\in I_\epsilon$, $n>n_{(\epsilon)}$.

Recall the definition of $\cal{C}$ and consider
\[
\cal{J}:=\frac{2+h_n(1+\widetilde F)}{2+\overline{h}_n(1-\widetilde
F)}\,.
\]
Then, we have
\begin{equation}\label{blue6}
\cal{BC}=\cal{B}(1+|\cal{J}|^2+2\Re(\xi
\bar{\cal{J}}))=\cal{B}+\cal{B}|\cal{J}|^2+2\cal{B}\Re(\xi \overline{\cal{J}})\,.
\end{equation}
Assume that $\theta\in I_\epsilon$ and $n>n_{(\epsilon)}$. Then, the first term is at most $C\epsilon^2$. For the third one, we use the formula for $\cal{B}$ and \eqref{defan} to get
\[
|2\cal{B}\Re(\xi \overline{\cal{J}})|\lesssim \sqrt{\cal{B}}|2+h_n(1+\widetilde F)|\lesssim \epsilon |2+h_n(1+2h_n^{-1})|\lesssim \epsilon\,,
\]
where at the last step we used an estimate for $h_n$ from Lemma \ref{lemma1}.
By \eqref{defan}, the second term in \eqref{blue6} can be rewritten as
\[
|2+h_n(1+\widetilde F)|^2=
|4+h_n|^2=16+8\Re h_n+|h_n|^2\,.
\]
Notice that, by the choice of $I_n$, 
\[
|8\Re h_n+|h_n|^2|\lesssim \epsilon\,,
\]
if $\theta\in I_\epsilon$ and $n>n_{(\epsilon)}$. To summarize, we
have that $\cal{A}$ deviates from $1$ by at most $C\epsilon$ and $\cal{B}\cal{C}$ deviates from $16$ by at most $C\epsilon$, provided that $\theta\in I_\epsilon$ and $n>n_{(\epsilon)}$. Thus,
\[
\sigma'(\theta)=\omega_n\left(1+O(\epsilon)\right)\,,
\]
if $\theta \in I_\epsilon$ and $n>n_{(\epsilon)}$. The positive parameter
$\omega_n$ depends on $n$ only and $\omega_n\sim 1$.\medskip

To apply Lemma \ref{obrez1} later, we have to check that $\sigma'$ satisfies the ``global lower and upper bounds'', i.e., we
 need to verify that
\[
 \cal{ABC}\sim 1\,,
\]
if $\theta\in \mathbb{T}$ and $n>n_{(\epsilon)}$. Indeed,  for $\cal{A}$, we use \eqref{blue7}. Then, to control $\cal{BC}$, we argue differently:
\begin{equation}\label{blue8}
\cal{BC}=\left|4+h_n+\xi\left( \overline{h}_n+2\left( 1-\frac{\overline{h}_n}{h_n}  \right)   \right)\right|^2\,.
\end{equation}
It is clear from Lemma \ref{lemma1} that $\cal{BC}\lesssim 1$. For the lower bound, we can use
an estimate (see again Lemma \ref{lemma1})
\[
\left| 1-\frac{\overline{h}_n}{h_n}  \right|\lesssim \epsilon
\]
to get
\[
4+h_n+\xi\left( \overline{h}_n+2\left( 1-\frac{\overline{h}_n}{h_n}  \right)   \right)=4+h_n+O(\epsilon)+\xi\left(\frac{\overline{h}_n}{h_n}-1\right)h_n+\xi h_n
\]
\[
=4+(\xi+1)h_n+O(\epsilon)\,.
\]
In these estimates, $O(\epsilon)$ is a shorthand for a function, whose absolute value is bounded by $C\epsilon$ for all $\theta\in [-\pi,\pi)$ and $n>n_{(\epsilon)}$.

Since $\arg(1+\xi)\in [-\pi/2,\pi/2]$ and $|\arg h_n|<C\epsilon$, we can say that there is $\epsilon_0>0$ such that
\[
|4+(\xi+1)h_n+O(\epsilon)|\gtrsim 1\,,
\]
provided that $\epsilon\in (0,\epsilon_0)$ and $\theta\in [-\pi,\pi), n>n_{(\epsilon)}$. Therefore, we  have
$
\sigma'\sim 1
$
for $\epsilon\in (0,\epsilon_0),n>n_{(\epsilon)}, \theta\in [-\pi,\pi)$.

\bigskip

To summarize, we showed that there is $\epsilon_0>0$ such that for every $\epsilon\in (0,\epsilon_0)$, there is $n_{(\epsilon)}\in \mathbb{N}$ and $I_\epsilon$, centered at the origin, such that for every $n>n_{(\epsilon)}$ there is a weight $\sigma'$ (density of absolutely continuous measure $\sigma$) for which 
\begin{itemize}
\item  $|\phi_{2n}(1,\sigma)|\gtrsim n^{\epsilon/2}$,

\item $\sigma'(\theta)\sim 1$ for $\theta\in [-\pi,\pi)$,

\item $\sigma'(\theta)=\omega_n(1+O(\epsilon))$ if $\theta\in I_\epsilon$, and $\omega_n\sim 1$.
\end{itemize}
Notice carefully that $\sigma$ is also $n$-dependent in this construction. The case of polynomials of odd degree can be handled by a minor modification of the argument (by working in $\mathbb{P}_{2n+1}$).\smallskip

Finally, we are in a position to use localization principle in Appendix B to produce a new weight whose  deviation from  the constant is smaller than $C\epsilon$ over the whole $\mathbb{T}$ and the value of the corresponding orthogonal polynomial at point $z=1$ is still large. To that end, consider $w_1$:
\[
w_1:=\left\{
\begin{array}{cc}
1, & \theta\notin I_\epsilon  \,, \\
\frac{\displaystyle \sigma'}{\displaystyle \min_{I_\epsilon}
\sigma'},& \theta\in I_\epsilon\,.
\end{array}\right.
\]
We have  $0\le w_1-1\lesssim \epsilon$ over $\mathbb{T}$.
Lemma \ref{obrez1} and \eqref{scale11} gives
\[
|\phi_{2n}(1,w_1)|>C_{(\epsilon)}|\phi_{2n}(1,\sigma)|>C_{(\epsilon)}n^{\epsilon/2}\,,
\]
if $n>n_{(\epsilon)}$.
The analogous estimate for the monic polynomials is true due to \eqref{fact12}. This proves \eqref{contra1} for even $n$, the case of odd $n$ can be handled similarly.\bigskip

\subsection{The case of large deviation, estimate \eqref{contra2}}

Let us introduce a parameter $\alpha\in (1/2,1)$ and define $\tau:=1-\alpha$. To clarify the meaning of these parameters, we now mention that  $T$ in \eqref{contra2} will be related to $\tau$ through an estimate $T\sim \tau^{-4}$. Notice that \eqref{contra2} will follow, if we show that there is $T_0>2$ such that for every $T>T_0$ there is $n_{(T)}$ so that 
\begin{equation}\label{contra27}
\sup_{1\le w\le T} \|\Phi_n(e^{i\theta},w)\|_\infty>C^{(2)}_{(T)}n^{\frac 12-C_6T^{-1/4}}
\end{equation}
for all $n>n_{(T)}$. In other words, we can assume that both $T$ and $n$ are ``large''. \smallskip

 Similarly to the
previous subsection, we apply Lemma \ref{decop}. In this Lemma, we take $2n$ instead of $n$ and the $(\ast)$ operation will act on $\mathbb{P}_{2n}$. The case of odd $n$ is similar.\smallskip

{\bf (a).} We define the function $H_n$ by
\begin{equation}\label{grey1}
 H_n= 2(1-e^{i\theta})^{\alpha}\ast \cal{K}_{[n/2]}\,,
\end{equation}
where $\cal{K}_l$ is Jackson kernel, i.e.,
\[
\cal{K}_l=c_l\cal{F}_l^2\sim l^{-3}\frac{\sin^4(lx/2)}{\sin^4(x/2)}\,.
\]
Recall that $\cal{K}_n\ge 0$, $\deg \cal{K}_n\le 2n-2$, and
$\|\cal{K}_n\|_{L^1[-\pi,\pi]}=1$.

We have $\Re H_n(z)>0$ for $z\in \mathbb{T}$ so $H_n$ has no zeros in $\overline{\mathbb{D}}$. Moreover, $H_n\in \mathbb{P}_{n-2}$ and 
\begin{equation}\label{aher}
H_n(0)=2\,.
\end{equation}
The estimates on $H_n$ and its limiting behavior when $n\to\infty$ are discussed in Appendix D.\bigskip

{\bf (b).} The function $\widetilde F$ is chosen by
\begin{equation}\label{nash-vybor}
\widetilde F:= 2H_n^{-1}\,.
\end{equation}


 We notice that, $\Re \widetilde F>0$ over
$\overline{\mathbb{D}}$ and $\Re \widetilde F$ is smooth on
$\mathbb{T}$. Due to \eqref{aher}, we also have $\widetilde F(0)=1$. Thus the
normalization condition in Lemma \ref{decop}
\[
\int_{-\pi}^\pi \Re\widetilde F d\theta=2\pi
\]
is satisfied.\bigskip

{\bf (c).} As in the case of ``small deviations'' we will define the polynomial $\phi_{2n}$ by specifying $\phi_{2n}^*$ first. Then, we will let $\phi_{2n}=(\phi_{2n}^*)^*$. Take $\phi_{2n}^*$ as
\[
\phi_{2n}^*(z):=\beta_n(Q_n+Q^*_n+Q_nH_n)\,,
\]
where $Q_n$ is defined as
\begin{equation}\label{grey6}
Q_n:=(1-z)^{-\alpha/2}\ast \cal{F}_n\,.
\end{equation}
The positive parameter $\beta_n$ is chosen to make sure that we have
 \begin{equation}\label{trebovanie1}
 \|(\phi_n^*)^{-1}\|_{L^2(\mathbb{T})}^2= \|\phi_n^{-1}\|_{L^2(\mathbb{T})}^2=2\pi\,,
  \end{equation}
  as required in \eqref{norma}.
 Notice that $Q_n\in \mathbb{P}_{n-1}$ and $\Re Q_n>0$ in $\overline{\mathbb{D}}$.\smallskip
 
 Since $(\ast)$ acts on $\mathbb{P}_{2n}$, $\deg (Q_n+Q^*_n+Q_nH_n)\le 2n$. Moreover, 
 \[
 Q_n(0)+Q_n^*(0)+Q_n(0)H_n(0)=3Q_n(0)>0\,.
 \]
 Therefore, $\deg \phi_{2n}=2n$ as required in Lemma \ref{decop}.\smallskip

  The absence of zeros of $\phi_{2n}^*$ in the unit disk can be proved as  in \eqref{bez-nuley}.\smallskip

Next, we proceed with estimating $\beta_n$.  
For the lower bound, we can follow \eqref{ocenka-sni} to write
\[
|Q_n+Q_n^*+Q_nH_n|^2\ge |Q_n|^2(\Re H_n)^2\gtrsim
\left\{
\begin{array}{cc}
\tau^2 n^{-\alpha}, & |\theta|<n^{-1}\,,\\
\tau^2|\theta|^{\alpha}, & |\theta|>n^{-1}\,,
\end{array}
\right.
\]
and the last inequality is due to \eqref{outside1}, \eqref{grey2}, and Lemma \ref{blue77}. This is true, provided that $\tau\in (0,\tau_0)$ and $n>n_{(\tau)}$.  For the upper bound, 
\[
|Q_n+Q_n^*+Q_nH_n|^2\lesssim |Q_n|^2\lesssim \left\{
\begin{array}{cc}
n^{\alpha}, & |\theta|<n^{-1}\,,  \\
|\theta|^{-\alpha}, & |\theta|>n^{-1}\,,
\end{array}
\right.
\]
again by Lemma \ref{grey4} and Lemma \ref{blue77}. From \eqref{trebovanie1},
\[
\beta_n^2\sim \int_{-\pi}^\pi |Q_n+Q_n^*+Q_nH_n|^{-2}d\theta\,,
\]
so
\[
\beta_n^2\lesssim \tau^{-2} n^{-\tau}+\int_{1/n}^1 \frac{1}{\tau^{2} \theta^{\alpha}}d\theta\lesssim \tau^{-3}
\]
for $\tau\in (0,\tau_0)$ and $n>n_{(\tau)}$. For the lower bound, we can write
\[
\beta^2_n\gtrsim n^{-\tau}+\int_{1/n}^1 |\theta|^{\alpha}d\theta\gtrsim 1\,.
\]
To summarize, we get an estimate
\begin{equation}\label{blue88}
1\lesssim \beta^2_n\lesssim \tau^{-3}\,,
\end{equation}
which holds true if $\tau\in (0,\tau_0)$ and $n>n_{(\tau)}$.\bigskip

Therefore,  for the value at
$z=1$, we get
\[
|\phi_{2n}(1)|=|\phi_{2n}^*(1)|\gtrsim C_{(\tau)} n^{\alpha/2}
\]
by Lemma \ref{blue77}, provided that $n>n_{(\tau)}$. This gives the growth required in \eqref{contra2} in terms of $\alpha$. 

We apply Lemma \ref{decop}. The identity \eqref{mp} provides the formula for $\sigma'$, a weight of orthogonality for our $\phi_{2n}(z)$.

We can rewrite \eqref{mp} as follows:
\begin{equation}\label{blue12}
\sigma'^{-1}=\upsilon_n \cal{ED}\,,
\end{equation}
where $\upsilon_n=\pi\beta_n^2/2$, and (due to \eqref{nash-vybor}),
\[
\cal{E}:=\frac{|Q_n|^2}{\Re \widetilde F},\,
\cal{D}:=\left|4+H_n+\xi\left( \overline{H}_n+2\left( 1-\frac{\overline{H}_n}{H_n}  \right)   \right)\right|^2,\,\xi=Q_n^*/Q_n=e^{i(2n\theta-2\Theta)}, \, \Theta=\arg Q_n\,.
\]
\smallskip

We will first obtain the lower and upper bounds for $\cal{E}$ and $\cal{D}$  over some interval $I_\tau$, centered at $\theta=0$. Then, we will handle all $\mathbb{T}$.\smallskip

Notice that
\[
\Re \widetilde  F=2\frac{\Re H_n}{|H_n|^2}\,.
\]
Consider $\cal{E}$. From Lemma \ref{grey4} and Lemma \ref{blue77}, we get
\begin{equation}\label{grey10}
 \cal{E}\sim  \tau^{-1} 
\end{equation}
if $\tau\in (0,\tau_0)$ and $|\theta|\in (n^{-1},\tau^2), n>n_{(\tau)}$.

Now, consider $\cal{D}$. For the upper bound, we obtain
\begin{equation}\label{dlyad}
\cal{D}\lesssim 1\,,
\end{equation}
if $\tau\in (0,\tau_0)$, and $\theta\in [-\pi,\pi), n>n_{(\tau)}$, as follows from \eqref{grey3}.
For the lower estimate, we notice that $|\xi|=1$ and write
\[
\left|4+H_n+\xi\left( \overline{H}_n+2\left( 1-\frac{\overline{H}_n}{H_n}  \right)   \right)\right|\ge 4-2\left|1-\frac{\overline{H}_n}{H_n}\right|-2|H_n|
\]
by triangle inequality.

If we define $u:=\arg H_n$, then 
\begin{equation}\label{sverha}
4-2\left|1-\frac{\overline{H}_n}{H_n}\right|=2\sqrt{2}(\sqrt 2-(1-\cos (2u))^{1/2})\gtrsim \tau^2\,,
\end{equation}
as follows from  \eqref{orange1}. From \eqref{grey3}, we then get
\[
 \cal{D}\ge (C_1\tau^2-C_2(|\theta|^\alpha+n^{-\alpha}))^2 \,,
\]
if $\tau\in (0,\tau_0), \theta\in (-\tau^2,\tau^2)$, and $n>n_{(\tau)}$. Now, we choose $I_\tau$ such that
\begin{itemize}
\item $I_\tau\subset [-\tau^2,\tau^2]$,
\item $C_2\max_{\theta\in I_\tau}|\theta|^\alpha<C_1\tau^2/2$.
\end{itemize}
That leads to an estimate
\begin{equation}\label{dlyad1}
\cal D\gtrsim \tau^4\,,
\end{equation}
which holds if  $\tau\in (0,\tau_0),  \theta\in I_\tau, n>n_{(\tau)}$.\smallskip

We collect  the estimates for $\cal{E}$ and $\cal{D}$ to get
\begin{equation}\label{put1}
\tau^3 \lesssim \cal{ED}\lesssim \tau^{-1}
\end{equation}
if $1/n<|\theta|, \theta\in I_\tau$.

Now, consider $\theta\in [-1/n,1/n]$ and write
\[
 4-2\left|1-\frac{\overline{H}_n}{H_n}\right|-2|H_n|\le 
\left|4+H_n+\xi\left( \overline{H}_n+2\left( 1-\frac{\overline{H}_n}{H_n}  \right)   \right)\right|\le 4+2\left|1-\frac{\overline{H}_n}{H_n}\right|+2|H_n|\,.
\]
Since $|H_n|\lesssim n^{-\alpha}$, the estimate \eqref{sverha} gives
\[
 4-2\left|1-\frac{\overline{H}_n}{H_n}\right|\lesssim  
\left|4+H_n+\xi\left( \overline{H}_n+2\left( 1-\frac{\overline{H}_n}{H_n}  \right)   \right)\right|\lesssim 1
\]
if $n>n_{(\tau)}$.
Taking the square of both sides and multiplying by $\cal{E}$ gives
\[
\frac{|Q_n|^2}{\Re H_n} ||H_n|-|\Im H_n||^2\lesssim \cal{ED}\lesssim \frac{|Q_n|^2|H_n|^2}{\Re H_n}
\]
and
\[
 \frac{n^{-2\alpha}\tau^3}{|H_n|^2} \sim n^{2\alpha}\tau^{-1} \frac{(\Re H_n)^4}{|H_n|^2}  \lesssim \cal{ED}\lesssim |H_n|^2n^{2\alpha}\tau^{-1}
\]
after Lemma \ref{grey4} and Lemma \ref{blue77} are applied. We apply Lemma \ref{grey4} to bound $|H_n|$ on the interval $[-1/n,1/n]$. Substituting $|H_n|\sim n^\tau|\theta|+\tau n^{-\alpha}$ and taking the minimum of the left hand side and maximum of the right hand side gives
\[
\tau^3\lesssim \cal{ED}\lesssim \tau^{-1}
\]
which matches \eqref{put1}.
Thus, to summarize, we get
\begin{equation}\label{sweet1}
\tau^3\lesssim \cal{ED}\lesssim \tau^{-1}
\end{equation}
if $\theta\in I_\tau$ and $n>n_{(\tau)}$.

 Next, we will prove the lower and upper estimates for $\cal{ED}$ which will be true for all of $\mathbb{T}$. We will need it to apply the localization principle. Notice that
\[
H_n\to 2(1-e^{i\theta})^\alpha, \, Q_n\to (1-e^{i\theta})^{-\alpha/2},\, \widetilde F\to (1-e^{i\theta})^{-\alpha}
\]
if $n\to\infty$ and convergence is uniform in $\theta\notin I_\tau$. Therefore,
\[
\cal{E}\to \frac{|1-e^{i\theta}|^{-\alpha}}{\Re ((1-e^{i\theta})^{-\alpha})}, \quad n\to\infty
\]
and convergence is uniform over $\theta\notin I_\tau$. So, away from $I_\tau$, the upper and lower bounds for the limiting function imply that $1\lesssim \cal{E}\lesssim \tau^{-1}$.
The upper estimate for $\cal D$ is easy: $|\cal{D}|\lesssim 1$ for all $\theta\in \mathbb{T}$. To obtain the lower bound, we recall again that $H_n\to 2(1-e^{i\theta})^\alpha$. In the formula for $\cal{D}$, we replace $H_n$ by its limiting value and consider the following function
\[
d_\delta:=\inf_{|\xi|=1,|z|=1, |1-z|>\delta}
\left|2+(1-z)^\alpha+\xi\left( \overline{(1-z)^\alpha}+\left( 1-\frac{\overline{(1-z)^\alpha}}{(1-z)^\alpha}  \right)   \right)\right|\,.
\]
Let us show that $d_\delta>0$ for positive $\delta$. For shorthand, introduce $t:=(1-z)^\alpha$ and write
\[
\left|2+t+\xi \frac{|t|^2+2i\Im t}{t}\right|=\left|\frac{2t+t^2+\xi (|t|^2+2i\Im t)}{t}\right|\ge \frac{|2t+t^2|-||t|^2+2i\Im t|}{|t|}=
\]
\[
 |t+2|- \left| |t|+2i\frac{\Im t}{|t|}\right|\ge \frac{|t|^2+4\Re t+4-|t|^2-4\sin^2(\arg t) }{ |t+2|+ \left| |t|+2i\frac{\Im t}{|t|} \right|  }\ge \frac{4\Re t}{ |t+2|+ \left| |t|+2i\frac{\Im t}{|t|} \right|  }\,.
\]
Since $|z|=1$ and $|1-z|\ge \delta$, we have $\Re t\gtrsim \delta^\alpha\tau^2$.

This implies
\[
1\gtrsim \cal{D}>C_{(\tau)}\,,
\]
if $\theta\notin I_\tau$ and $n>n_{(\tau)}$. Combining the bounds, we get
\[
c^{(1)}_{(\tau)}<\cal{ED}<c^{(2)}_{(\tau)}
\]
for $\theta\in [-\pi,\pi)$ and $n>n_{(\tau)}$. Considering the estimates on $\beta_n$ and formula \eqref{blue12}, we obtain 
\begin{equation}\label{bolshaya12}
C^{(1)}_{(\tau)}<\sigma'(\theta)<C^{(2)}_{(\tau)}
\end{equation}
for $\theta\in [-\pi,\pi)$.

\bigskip

To summarize, we proved that there is $\tau_0>0$ such that for all $\tau\in (0,\tau_0)$, there is 
$n_{(\tau)}\in \mathbb{N}$ and interval $I_\tau$ centered at $\theta=0$ so that for every $n>n_{(\tau)}$ there a weight $\sigma'$ (defining the absolutely continuous measure $\sigma$) such that

\begin{itemize}
\item  $|\phi_{2n}(1,\sigma)|\gtrsim C_{(\tau)}n^{\alpha/2}$.

\item The bound holds 
\begin{equation}\label{bolshaya}
C^{(1)}_{(\tau)}<\sigma'(\theta)<C^{(2)}_{(\tau)}
\end{equation}
for $\theta\in [-\pi,\pi)$.

\item $ \upsilon_n^{-1}\tau\lesssim  \sigma'(\theta) \lesssim \upsilon_n^{-1}\tau^{-3}$ for $\theta\in I_\tau$.
\end{itemize}
The last property follows from representation $\sigma'=\upsilon_n^{-1}(\cal{ED})^{-1}$ and \eqref{sweet1}.\bigskip

 It is only left to use the localization principle. Since we have \eqref{bolshaya}, Lemma \ref{obrez1} is applicable.
Consider $w_1$:
\[
w_1=\left\{
\begin{array}{cc}
1, & \theta\notin I_\tau,\\
\frac{\displaystyle \sigma'}{\displaystyle \min_{\theta\in I_\tau}
\sigma'},& \theta\in I_\tau\,.
\end{array}\right.
\]
We have  $1\le w_1\le T$ uniformly over $\mathbb{T}$ and $T:=\max_{I_\tau} \sigma'/\min_{I_\tau} \sigma'\lesssim \tau^{-4}$.
Lemma \ref{obrez1} and \eqref{scale11} give
\[
|\phi_{2n}(1,w_1)|>C^{(3)}_{(\tau)}|\phi_{2n}(1,\sigma)|>C^{(4)}_{(\tau)}n^{\alpha/2}, \quad
n>n_{(\tau)}\,.
\]
The estimate \eqref{contra2} now follows from \eqref{fact11} and \eqref{scale11}.
The case of odd $n$ can be handled similarly.

\end{proof}

{\bf Remark.} We can reformulate the Theorem \ref{mainn} in a different way.
Let us introduce two functions defined for $t\ge 0$:
\[
\overline\nu(t):=\limsup_{n\to\infty} \frac{   \log \sup_{1\le w\le 1+t} \|\Phi_n(z,w)\|_\infty}{\log n}
\]
and
\[
\underline\nu(t):=\liminf_{n\to\infty} \frac{   \log \sup_{1\le w\le 1+t} \|\Phi_n(z,w)\|_\infty}{\log n}\,.
\]
Then the following two estimates are immediate from combining Theorem~\ref{fodd} and Theorem~\ref{mainn}:
\[
\frac 12-\frac{c_1}{ t^{1/4}}\le \underline{\nu}(t) \le \overline{\nu}(t) \le \frac 12-\frac{c_2}{t}, \quad t>2
\]
and
\[
c_3(t-1)\le \underline{\nu}(t)\le \overline{\nu}(t)\le c_4(t-1), \quad 1<t<2\,.
\]
where $c_1,c_2,c_3,c_4$ are some positive absolute constants.
\bigskip

The Theorem \ref{mainn} can now  be used to show that for some $w$ that satisfies $w,w^{-1}\in L^\infty(\mathbb{T})$  the sequence $\|\phi_n(e^{i\theta},w)\|_{L^\infty(\mathbb{T})}$ can grow in $n$ at a certain rate.

\begin{theorem}
Given $\epsilon\in (0,1]$, there is a weight $w$ and a subsequence $\{k_n\}\subset \mathbb{N}$ that satisfy $1\le w\le 1+\epsilon$ and
\[
 \|\phi_{k_n}(e^{i\theta},w)\|_{L^\infty(\mathbb{T})}\gtrsim  k_n^{c_5\epsilon}\,.
\]
If $T\ge 2$, then there is a weight $w$ and a subsequence $\{k_n\}\subset \mathbb{N}$ that satisfy $1\le w\le T$ and
\[
\|\phi_{k_n}(e^{i\theta},w)\|_{L^\infty(\mathbb{T})}\gtrsim k_n^{1/2-c_6T^{-1/6}}\,.
\]
Here $c_5$ and $c_6$ denote positive absolute constants.
\end{theorem}
\begin{proof} We give the proof only for the first case when the deviation of the weight from the constant is small. The other case of large deviation can be handled similarly.

We consider the sequence of disjoint arcs $\{I_n\}$ in $\mathbb{T}$ that satisfies the following condition
\[
\sum_{n=1}^\infty \delta_n <2\pi, \quad |I_n|=\delta_n\,.
\]
Denote the centers of these arcs by $\{e^{i\theta_n}\}$. Take a monotonically increasing sequence $\{k_n\}\subset \mathbb{N}$ so that
\begin{equation}\label{gena1}
C^{(1)}_{(\epsilon)}\delta_n^{-1}k_n^{C_5\epsilon}>k_n^{C_5\epsilon/2}, \quad n=1,2,\ldots
\end{equation}
and  $C^{(1)}_{(\epsilon)},C_5$ are taken from \eqref{contra1}. Then, by Theorem \ref{mainn}, we have a sequence of weights $\{w_n\}$  so that
\[
1\le w_n\le 1+\epsilon
\]
and
\begin{equation}\label{gena2}
|\phi_{k_n}(1,w_n)|>C^{(1)}_{(\epsilon)}k_n^{C_5\epsilon}\,.
\end{equation}
Now, we define  weight $w$ which is equal to $1$ away from $\bigcup_{n=1}^\infty I_n$ and
\[
w(\theta)=w_n(\theta-\theta_n)
\]
on each $I_n$ (we recall that $e^{i\theta_n}$ is the center of $I_n$). We clearly have $1\le w\le 1+\epsilon$.
The localization principle \eqref{eq1} now implies that
\[
|\phi_{k_n}(e^{i\theta_n},w)|\gtrsim \delta_n^{-1}|\phi_{k_n}(1,w_n)|\,.
\]
The application of \eqref{gena1} and \eqref{gena2} now finishes the proof.

\end{proof}

\bigskip
\section{Appendix A: Method used to prove Theorem \ref{T3-i} }

In this section we explain an idea used in the proof of Theorem
\ref{T3-i}. This material is taken from \cite{den-jat}. We start
with recalling some basic facts about the polynomial orthogonal on
the unit circle. With any probability measure $\mu$, which is
defined on the unit circle and have infinitely many growth points,
one can associate the orthonormal polynomials of the first and
second kind, $\{\phi_n\}$ and $\{\psi_n\}$, respectively.
$\{\phi_n\}$ satisfy the following recursions \mbox{(\cite{sim1}, p. 57)}
with Schur parameters $\{\gamma_n\}$:
\begin{equation}\left\{\begin{array}{cc}
\phi_{n+1}=\rho_n^{-1}(z\phi_n-\overline\gamma_n\phi_n^*),&\phi_0=1\,,\\
\phi_{n+1}^*=\rho_n^{-1}(\phi_n^*-\gamma_n z\phi_n),&\phi^*_0=1\,,
\end{array}\right.
\label{srecurs}
\end{equation}
and $\{\psi_n\}$ satisfy the same recursion but with Schur
parameters $\{-\gamma_n\}$, i.e.,
\begin{equation}\label{secon}\left\{\begin{array}{cc}
\psi_{n+1}=\rho_n^{-1}(z\psi_n+\overline\gamma_n\psi_n^*),&\psi_0=1\,,\\
\psi_{n+1}^*=\rho_n^{-1}(\psi_n^*+\gamma_n z\psi_n),&\psi^*_0=1\,.
\end{array}\right.
\end{equation}
The coefficient $\rho_n$ is defined as
\[
\rho_n:=\sqrt{1-|\gamma_n|^2}\,.
\]
The following Bernstein-Szeg\H{o} approximation  is valid:
\begin{lemma}{\rm (\cite{5},\cite{sim1})}\, Suppose $d\mu$ is a probability measure and $\{\phi_j\}$ and $\{\psi_j\}$ are the
corresponding orthonormal polynomials of the first/second kind,
respectively. Then, for any $N$, the Caratheodory function
\[
F_N(z)=\frac{\psi_N^*(z)}{\phi_N^*(z)}=\int_{\mathbb T}
C(z,e^{i\theta})d\mu_N(\theta),\,\,{\rm where}\quad
d\mu_N(\theta)=\frac{d\theta}{2\pi|\phi_N(e^{i\theta})|^2}=
\frac{d\theta}{2\pi|\phi^*_N(e^{i\theta})|^2}
\]
has the first $N$ Taylor coefficients  identical to the Taylor
coefficients of the function
\[
F(z)=\int_{\mathbb T}C(z,e^{i\theta})d\mu(\theta)\,\,.
\]
In particular, the polynomials $\{\phi_j\}$ and $\{\psi_j\}$,
$j\!\le\!N$ are the orthonormal polynomials of the first/second kind
for the measure $d\mu_N$.
\end{lemma}
We also need the following Lemma which can be verified directly:
\begin{lemma}\label{vspomag}
The polynomial $P_n(z)$ of degree $n$ is the orthonormal polynomial
for a probability measure with infinitely many growth points if and
only if
\begin{itemize}
\item[1.] $P_n(z)$ has all $n$ zeroes inside $\mathbb D$ (counting the multiplicities).
\item[2.] The normalization conditions
\[
\int_\mathbb
T\frac{d\theta}{2\pi|P_n(e^{i\theta})|^2}=1~,\quad\operatorname{coeff}(P_n,n)>0
\]
are satisfied.\end{itemize}\end{lemma}
\begin{proof} Take $2\pi|P_n(e^{i\theta})|^{-2}d\theta$ itself as a
probability measure. The orthogonality is then immediate.
\end{proof}

We continue with a Lemma which paves the way for constructing the
measure giving, in particular, the optimal bound \eqref{osnova-i}.
It is a special case of a solution to the
 truncated moment problem. In this Lemma, $\phi_n^*$ is defined as an application of $(\ast)$ operation to $\phi_n$, considered as an element of $\mathbb{P}_n$.

\begin{lemma}\label{decop}
Suppose we are given a polynomial $\phi_n$ and  Caratheodory
function $\~F$
 which satisfy the following properties
\begin{itemize}
\item[1.] $\deg \phi_n=n$ and $\phi_n^*$ has no roots in $\overline{\mathbb D}$. 
\item[2.] Normalization on the size and ``rotation\!"
\begin{equation}\label{norma}
\int_\mathbb T|\phi_n^*(e^{i\theta})|^{-2}d\theta
=2\pi~,\quad\phi_n^*(0)>0\,\,.
\end{equation}

\item[3.] $\~F\!\in\!C^\infty(\mathbb T)$, $\Re\~F>0$ on $\mathbb T$, and
\begin{equation}\label{norka}
\frac1{2\pi}\int_\mathbb T\Re\~F(e^{i\theta})d\theta=1\,\,.
\end{equation}

\end{itemize}
Consider two probability measures $\mu_n$ and $\~\sigma$ given by
\[
d\mu_n:=\frac{d\theta}{2\pi|\phi_n^*(e^{i\theta})|^2}, \quad
d\~\sigma=\~\sigma'd\theta:=\frac{\Re \~F(e^{i\theta})}{2\pi}d\theta
\]
and denote  their Schur (recursion) coefficients by
 $\{\gamma_j\}$ and $\{\~\gamma_j\}$, respectively. Then, the
probability measure $\sigma$, corresponding to Schur (recursion) coefficients
\[
\gamma_0,\ldots, \gamma_{n-1},\~\gamma_0,\~\gamma_1,\ldots
\]
is purely absolutely continuous with the density (weight) and
\begin{equation}\label{mp}
\sigma'=\frac{4\~\sigma'}{|\phi_n+\phi_n^*+\~F(\phi_n^*-\phi_n)|^2}=\frac{2\Re\~F}
{\pi|\phi_n+\phi_n^*+\~F(\phi_n^*-\phi_n)|^2}\,\,.
\end{equation}
The polynomial $\phi_n$ is the orthonormal polynomial for $\sigma$.
\end{lemma}
The proof of this Lemma is contained in \cite{adt}. We, however,
prefer to give its sketch  here.

\begin{proof}
First, notice that $\{\~\gamma_j\}\in \ell^1$ by Baxter's Theorem
(see, e.g., \cite{sim1}, Vol.1, Chapter 5). Therefore, $\sigma$ is
purely absolutely continuous by the same Baxter's criterion. Define
the orthonormal polynomials of the first/second kind corresponding
to measure $\~\sigma$ by $\{\~\phi_j\}, \{\~\psi_j\}$. Similarly,
let $\{\phi_j\}, \{\psi_j\}$ be orthonormal polynomials for
$\sigma$. Since, by construction, $\mu_n$ and $\sigma$ have
identical first $n$ Schur parameters, $\phi_n$ is $n$-th orthonormal
polynomial for~$\sigma$.

 Let
us compute the polynomials $\phi_j$ and $\psi_j$, orthonormal with
respect to $\sigma$, for the indexes $j>n$. By \eqref{secon}, the
recursion can be rewritten in the following matrix form
\begin{equation}\label{m-ca}
\left(\begin{array}{cc}
\phi_{n+m} & \psi_{n+m}\\
\phi_{n+m}^* & -\psi_{n+m}^*
\end{array}\right)=\left(\begin{array}{cc}
{\cal A}_m & {\cal B}_m\\
{\cal C}_m & {\cal D}_m
\end{array}\right)\left(\begin{array}{cc}
\phi_{n} & \psi_{n}\\
\phi_{n}^* & -\psi_{n}^*
\end{array}\right)\end{equation}
where ${\cal A}_m, {\cal B}_m, {\cal C}_m, {\cal D}_m$ satisfy
\begin{eqnarray*}\left(\begin{array}{cc}
{\cal A}_0 & {\cal B}_0\\
{\cal C}_0 & {\cal D}_0
\end{array}\right)=\left(\begin{array}{cc}
1 & 0\\
0 & 1
\end{array}\right),\hspace{6cm}\\
\left(\begin{array}{cc}
{\cal A}_m & {\cal B}_m\\
{\cal C}_m & {\cal D}_m
\end{array}\right)=\frac1{\~\rho_0\cdot\ldots\cdot\~\rho_{m-1}}\left(\begin{array}{cc}
z & -\overline{\~\gamma}_{m-1}\\
-z\~\gamma_{m-1} & 1
\end{array}\right)\cdot\ldots\cdot\left(\begin{array}{cc}
z & -\overline{\~\gamma}_0\\
-z\~\gamma_0 & 1
\end{array}\right)\end{eqnarray*}
and thus depend only on $\~\gamma_0,\ldots,\~\gamma_{m-1}$.
Moreover, we have
\[\left(\begin{array}{cc}
\~\phi_m & \~\psi_m\\
\~\phi_m^* & -\~\psi^*_m
\end{array}\right)=\left(\begin{array}{cc}
{\cal A}_m & {\cal B}_m\\
{\cal C}_m & {\cal D}_m
\end{array}\right)\left(\begin{array}{cc}
1 & 1\\
1 & -1
\end{array}\right)
\,\,.
\]
Thus, ${\cal A}_m\!=\!(\~\phi_m\,{+\,\~\psi_m)/2,~{\cal
B}_m\!=\!(\~\phi_m\,-}\,\~\psi_m)/2,~ {\cal
C}_m\!=\!(\~\phi^*_m\,{-\,\~\psi^*_m)/2,~{\cal
D}_m\!=\!(\~\phi^*_m\,+}\,\~\psi^*_m)/2$ and their substitution into
\eqref{m-ca} yields
\begin{equation}\label{intert}
2\phi_{n+m}^*=\phi_n(\~\phi_m^*-\~\psi^*_m)+\phi_n^*(\~\phi_m^*+\~\psi^*_m)=
\~\phi_m^*\left(\phi_n+\phi_n^*+\~
F_m(\phi_n^*-\phi_n)\right)\,,\end{equation} where
\[
\~F_m(z)=\frac{\~\psi^*_m(z)}{\~\phi^*_m(z)}\,\,.
\]
Since $\{\~\gamma_n\}\!\in\!\ell^1$ and $\{\gamma_n\}\!\in\!\ell^1$,
we have (\cite{sim1}, p.~225)
\[
\~F_m\to\~F~{\rm as~}m\to\infty~{\rm and~}
\phi_j^*\to\Pi,~\~\phi_j^*\to\~\Pi~{\rm as~}j\to\infty\,\,.
\]
uniformly on $\overline{\mathbb D}$. The functions $\Pi$ and $\~\Pi$
are the Szeg\H{o} functions of $\sigma$ and $\~\sigma$,
respectively, i.e., they are the outer functions in $\mathbb D$ that
satisfy
\begin{equation}\label{facti}
|\Pi|^{-2}=2\pi\sigma',\quad |\~\Pi|^{-2}=2\pi\~\sigma'\,\,
\end{equation}
on $\mathbb{T}$. In \eqref{intert}, send $m\to\infty$ to get
\begin{equation}\label{facti1}
2\Pi=\~\Pi\left(\phi_n+\phi_n^*+\~F(\phi_n^*-\phi_n)\right)
\end{equation}
and we have \eqref{mp} after taking the square of absolute values
and using \eqref{facti}.
\end{proof}

\bigskip

\section{Appendix B: the Localization Principle}
Given a weight $w$ on $[-\pi,\pi]$, we define
\[
\lambda(w):=\exp\left(\frac{1}{4\pi}\int_{\mathbb{T}} \log(2\pi
w(\theta)) d\theta\right), \quad
\Lambda(w):=\sqrt{\|w\|_{L^1(\mathbb{T})}}\,.
\]
The following Theorem was proved in \cite{adt} (Theorem 5.1).
\begin{theorem}{\rm (\cite{adt})}\label{lemma21}
Let $w_1$ and $w_2$ be two weights on $[-\pi,\pi]$ so that
\begin{equation}\label{c2}
w_1(\theta)=w_2(\theta), \quad \theta\in [-\epsilon,\epsilon]\,.
\end{equation}
Then
\begin{equation}\label{localization}
\left|\frac{\phi_n(1,w_1)}{\phi_n(1,w_2)}\right|\le
\frac{\Lambda(w_2)}{\lambda(w_1)}+\frac{4\Lambda(w_1)}{\epsilon\lambda(w_1)}
\left(\int_{|\theta|>\epsilon}|\phi_n(e^{i\theta},w_1)\phi_n(e^{i\theta},w_2)|(w_1+w_2)d\theta\right)
\end{equation}
for all $n$.
\end{theorem}
We get a simple corollary.
\begin{lemma} \label{obrez1} Assume that
\begin{equation}\label{oc1}
 0<m_1\le
w_{1}(\theta)\le m_2,\,\, 0<m_1\le
w_{2}(\theta)\le m_2\,\, {\it for} \,\,\, \theta\in
[-\pi,\pi)
\end{equation}
and
\begin{equation}\label{c8}
w_1(\theta)=w_2(\theta) \,\, {\it for} \,\,\, \theta\in
[-\epsilon,\epsilon]\,.
\end{equation}
Then,
\begin{equation}\label{eq1}
\frac{\epsilon m_1}{m_2}\lesssim
\left|\frac{\phi_n(1,w_1)}{\phi_n(1,w_2)}\right|\lesssim
\frac{m_2}{\epsilon m_1}\,.
\end{equation}

\end{lemma}
\begin{proof}
The lower and upper bounds for the weights imply that 
\[
\sqrt{m_1}\lesssim  \Lambda\lesssim \sqrt{m_2}, \quad \sqrt{m_1}\lesssim  \lambda\lesssim \sqrt{m_2}
\] 
for both weights.

Notice that, e.g.,
\[
\int_{-\pi}^{\pi}
|\phi_n(e^{i\theta},w_1)\phi_n(e^{i\theta},w_2)|w_1d\theta\le
\left(\int_{-\pi}^{\pi}|\phi_n(e^{i\theta},w_2)|^2w_1d\theta\right)^{1/2}
\]
by Cauchy-Schwarz and definition of $\phi_n(e^{i\theta},w_1)$. Since
\[
1=\int_{-\pi}^{\pi}|\phi_n(e^{i\theta},w_2)|^2w_2d\theta\ge m_1
\int_{-\pi}^{\pi}|\phi_n(e^{i\theta},w_2)|^2d\theta
\]
we get
\[
\int_{-\pi}^{\pi}
|\phi_n(e^{i\theta},w_1)\phi_n(e^{i\theta},w_2)|w_1d\theta\le
\left(\frac{m_2}{m_1}\right)^{1/2}\,.
\]
To prove the lower bound in \eqref{eq1},  it is sufficient to
change the roles of $w_1$ and $w_2$.
\end{proof}

\bigskip

\section{Appendix C: The properties of auxiliary polynomials, I}

In this Appendix, we study the properties of the polynomials
introduced in subsection 3.1. Recall that
\begin{equation}\label{svertka1}
h_n=2(1-e^{i\theta})^\epsilon\ast \cal{F}_n, \quad \widetilde F=2h_n^{-1} \,.
\end{equation}
We suppress the dependence of  $\widetilde F$ on $n$ to make notation consistent with Lemma \ref{decop}.

The standard properties of convolution (Theorem 8.14, \cite{folland}) yield the following Lemma.
\begin{lemma}\label{frukt}
For  fixed $\delta\in (0,\pi)$ and $\epsilon\in (0,1)$, we have
\[
\lim_{n\to\infty} h_n(\theta)= 2(1-e^{i\theta})^{\epsilon}, \quad \lim_{n\to\infty}\widetilde F(e^{i\theta})=
(1-e^{i\theta})^{-\epsilon}
\]
uniformly over $\{\theta: |\theta|\in [\delta,\pi]\}$.
\end{lemma}\smallskip

{\bf Remark.} Notice that $|\arg (1-e^{i\theta})^\epsilon|\le \epsilon\pi/2$ and
\begin{equation}\label{boio}
|\Im (1-e^{i\theta})^\epsilon|\lesssim  \epsilon \Re (1-e^{i\theta})^\epsilon\,.
\end{equation}

 From that point until the end of this Appendix, we  assume that $\epsilon\in (0,\epsilon_0)$ where $\epsilon_0\in (0,1)$ is a fixed constant to be chosen below.  Let us start with the following simple
observations about $(1-e^{i\theta})^\epsilon$, the function we are approximating. Taylor expansion for $e^{i\theta}$ around the origin is $e^{i\theta}=1+i\theta-\theta^2/2+O(|\theta|^3), |\theta|<\pi$. Substituting this into $(1-e^{i\theta})^\epsilon$ gives
\[
(1-e^{i\theta})^\epsilon=(-i\theta)^\epsilon (1+i\theta/2+O(\theta^2))^\epsilon\,.
\]
Using Taylor expansion of the logarithm around point $1$, we can find an absolute constant $\delta_0>0$ such that for every $\theta\in (-\delta_0,\delta_0)$, we get
\[
(-i\theta)^\epsilon (1+i\theta/2+O(\theta^2))^\epsilon=|\theta|^\epsilon e^{-i\epsilon \cdot {\rm sgn} (\theta) \cdot \pi/2}e^{\epsilon\log(1+i\theta/2+O(\theta^2))}=
\]
\[
|\theta|^\epsilon \exp(-i\epsilon \cdot {\rm sgn} (\theta) \cdot \pi/2+i\epsilon\theta/2)(1+O(\epsilon\theta^2))\,.
\]
For $\theta: |\theta|>\delta_0$, we can write
\[
(1-e^{i\theta})^\epsilon=e^{\epsilon \log(1-e^{i\theta})}=1+O(\epsilon)\,.
\]
These two expansions imply, in particular, that for all $\theta\in [-\pi,\pi)$, we can write the formula for the real part in the following compact form
\begin{eqnarray}\label{arom1}
\Re (1-e^{i\theta})^{\epsilon}=|\theta|^{\epsilon}\cos(\epsilon
\pi/2)(1+O(\epsilon
|\theta|))\,.
\end{eqnarray}
The bound \eqref{boio} then yields
\begin{equation}\label{arom2}
|\Im (1-e^{i\theta})^{\epsilon}|\lesssim \epsilon |\theta|^{\epsilon}\,.
\end{equation}

\smallskip

We will need the following representation for the Fejer kernel
\begin{equation}\label{arom3}
\cal{F}_n:=\frac{1}{2\pi n}\left(  \frac{\sin(n\theta/2)}{\sin(\theta/2)}        \right)^2
=\frac{2}{\pi n}\left(\frac{\sin^2(n\theta/2)}{\theta^2}+O(1)\right)
\end{equation}
for $|\theta|<\pi$ and $n\in \mathbb{N}$.  This expansion follows immediately from the Taylor expansion of $\sin(\theta/2)$ around the origin. We recall that
\begin{equation}\label{normalek}
\int_{-\pi}^\pi \cal{F}_n(\theta)d\theta=1\,.
\end{equation}
 \smallskip

In the proof of the next Lemma, the remainder terms
in the formulas \eqref{arom1} and \eqref{arom3}, e.g., $\epsilon O(|\theta|^{1+\epsilon})$ and $n^{-1}O(1)$ contribute the
smaller order terms which can be neglected. This will be explained in detail.

\begin{lemma}\label{lemma1}
For every $\epsilon\in (0,\epsilon_0)$, there is
$n_{(\epsilon)}\in \mathbb{N}$ such that
\begin{eqnarray}\label{aarr}
 |\arg h_n(\theta)|\lesssim \epsilon\,,
 \\ \label{aarr1}
\Re h_n(\theta) \sim n^{-\epsilon}+|\theta|^\epsilon\,,
\\
   |h_n(\theta)|\sim n^{-\epsilon}+|\theta|^\epsilon \label{aarr2}
\end{eqnarray}
for  $\theta\in [-\pi,\pi]$
and all $n>n_{(\epsilon)}$.
\end{lemma}
\begin{proof}
We start by noticing that, since the Fejer kernel and $\Re (1-e^{i\theta})^\epsilon$ are nonnegative, 
\begin{equation}\label{pokoren}
|\Im h_n|\lesssim  \epsilon \Re h_n
\end{equation}
as follows from \eqref{boio}. Thus, we have \eqref{aarr} since $|\tan \arg h_n|\lesssim \epsilon$.\smallskip

To estimate $\Re h_n$, we can substitute \eqref{arom1} and \eqref{arom3} into \eqref{svertka1}. For $\theta: |\theta|<\pi$, we have
\begin{equation}\label{red1}
\Re h_{n}(\theta)=\frac{2}{\pi}\cos(\epsilon\pi/2)h_n^{(1)}(\theta)\,,
\end{equation}
where
\begin{equation}\label{perm1}
h_n^{(1)}(\theta):= n^{-1}\int_{-\pi}^{\pi} |x|^\epsilon
\frac{\sin^2(n(x-\theta)/2)}{(x-\theta)^2}dx+\epsilon_{(n,\theta)}\,.
\end{equation}
Let us focus on the main term now. We change variables $\widehat x:=nx, \widehat \theta:=n\theta$ to write
\begin{equation}\label{blue5}
n^{-1}\int_{-\pi}^{\pi} |x|^\epsilon
\frac{\sin^2(n(x-\theta)/2)}{(x-\theta)^2}dx=n^{-\epsilon}
M_n(\widehat \theta),\quad ,\quad M_n(\widehat \theta):= \int_{-\pi n}^{\pi n} |t|^\epsilon
g(\widehat \theta-t)dt
\end{equation}
with
\[
g(t):=\frac{\sin^2(t/2)}{t^2}\,.
\]
The estimation of the last integral shows that
\[
M_n(\widehat \theta)\sim 1+|\widehat \theta|^\epsilon
\]
for $\widehat \theta\in [-\pi n,\pi n]$. Indeed, given any $a>1$, we can write
\begin{equation}\label{proshloe}
\int_{-a}^a |t|^\epsilon g(\xi-t)dt=I_1+I_2, 
\end{equation}
\[
 I_1:=\int_{t\in [-a,a], |\xi-t|<|\xi|/2} |t|^\epsilon g(\xi-t)dt, \, I_2:=\int_{t\in [-a,a], |\xi-t|>|\xi|/2} |t|^\epsilon g(\xi-t)dt\,.
\]
Both $I_1$ and $I_2$ are nonnegative and, considering $\xi: |\xi|\in [1,a]$, we have
\[
I_1\sim |\xi|^\epsilon,\, I_2\lesssim \int_{|t-\xi|>|\xi|/2} \frac{|t-\xi|^\epsilon}{|t-\xi|^2}dt\lesssim 1\,.
\] 
To get the last estimate, we used trivial observation that $|t|\lesssim |t-\xi|$ provided that $|t-\xi|>|\xi|/2$. For $\xi: |\xi|<1$, we get
\[
I_1+I_2\sim 1\,.
\]

We are left with controlling $\epsilon_{(n,\theta)}$. We can write
\begin{equation}\label{esr}
|\epsilon_{(n,\theta)}|\lesssim  n^{-1}+\epsilon n^{-1} \int_{-\pi}^{\pi}
\frac{\sin^2(n(x-\theta)/2)}{(x-\theta)^2 }
|x|^{1+\epsilon}dx 
\lesssim n^{-1}+\epsilon n^{-1} \int_{-\pi}^{\pi}
\frac{\sin^2(n(x-\theta)/2)}{\sin^2((x-\theta)/2) }
|x|^{1+\epsilon}dx \,.
\end{equation}
Now we use  \eqref{normalek} to obtain
\[
\epsilon_{(n,\theta)}\lesssim 
 n^{-1}+\epsilon  |\theta|^{1+\epsilon}+\epsilon n^{-1}
\int_{-\pi}^{\pi} \frac{\sin^2(n(x-\theta)/2)}{\sin^2((x-\theta)/2) }
\left||\theta|^{1+\epsilon}-|x|^{1+\epsilon}\right|dx\lesssim 
 n^{-1}+\epsilon
|\theta|^{1+\epsilon}+\epsilon n^{-1}\log n \,.
\]
To estimate the last integral,  we used the bound
$||x|^{1+\epsilon}-|\xi|^{1+\epsilon}|\lesssim |x-\xi|$ which is true for  $|x|,|\xi|<\pi$. Notice that
\[
|\epsilon_{(n,\theta)}/{(n^{-\epsilon}M_n)}|\lesssim \epsilon
\]
if $\theta\in [-\pi,\pi], n>n_{(\epsilon)}$.
Thus, we showed that 
\begin{equation}\label{tockta}
\Re h_n=\frac{2}{\pi}  \cos(\epsilon\pi/2)\cdot n^{-\epsilon} M_n(n\theta) \cdot (1+O(\epsilon))\,,  
\end{equation}
if $\theta\in [-\pi,\pi)$ and $n>n_{(\epsilon)}$. This implies \eqref{aarr1}.
Application of the estimate \eqref{pokoren} gives \eqref{aarr2}.

\end{proof}

{\bf Remark.} The proof given above provides the following representation
\begin{equation}\label{blue1}
h_n(\theta)=\frac{2}{\pi}\cos(\epsilon\pi/2)\cdot n^{-\epsilon} M_n(n\theta) \cdot (1+O(\epsilon))
\end{equation}
for  $\theta\in [-\pi,\pi]$ and  $n>n_{(\epsilon)}$.\bigskip

 The derivative of function  $M_n(t)$, introduced in \eqref{blue5}, is controlled in the following Lemma.
\begin{lemma}
We have
\begin{equation}\label{blue2}
|M_n'(t)|\lesssim  n^{\epsilon-2} +\epsilon (1+|t|)^{\epsilon-1}
\end{equation}
for $|t|<\pi n/2$.
\end{lemma}
\begin{proof} Differentiating the formula for $M_n$, we get 
\[
|M_n'(t)|=\left|\int_{-\pi n}^{\pi n} |s|^\epsilon g'(s-t)ds\right|
\]
\[
\lesssim  n^{\epsilon} \left( \frac{1}{(\pi n-t)^2+1}+ \frac{1}{(\pi n+t)^2+1}  \right)+\epsilon \left|\int_{-\pi n}^{\pi n} |s|^{\epsilon-1}  {\rm sgn} (s) \cdot g(s-t)ds\right|\,.
\]
Consider the last integral. We can rewrite it as
\[
\int_{-\pi n}^{\pi n} |s|^{\epsilon-1} {\rm sgn} (s) \cdot g(s-t)ds=\int_{0}^{1} s^{\epsilon-1}  \Bigl( g(s-t)-g(-s-t)\Bigr)ds+\int_{1<|s|<\pi n} |s|^{\epsilon-1} {\rm sgn} (s)  \cdot g(s-t)ds\,.
\]
Since $|g'(t)|\lesssim (1+t^2)^{-1}$, the Mean Value Theorem applied to the integrand in the first integral gives
\[
\left|\int_{0}^{1} s^{\epsilon-1}  \Bigl( g(s-t)-g(-s-t)\Bigr)ds\right|\lesssim  (1+t^2)^{-1}\int_0^1 s^{\epsilon-1}sds\lesssim (1+t^2)^{-1}\,.
\]
For the second integral, we have
\[
\left|\int_{1<|s|<\pi n} |s|^{\epsilon-1}  \cdot {\rm sgn} (s) \cdot g(s-t)ds\right|\lesssim \int_{\mathbb{R}} (1+|s|)^{\epsilon-1} \frac{1}{(t-s)^2+1}ds\lesssim (1+|t|)^{\epsilon-1}\,.
\]
At the last step, we argued similarly to \eqref{proshloe}.

Thus, we have
\[
|M_n'(t)|\lesssim  n^{\epsilon} \left( \frac{1}{(\pi n-t)^2+1}+ \frac{1}{(\pi n+t)^2+1}  \right)+\epsilon (1+|t|)^{\epsilon-1}
\]
and this finishes the proof.
\end{proof}

\begin{lemma}\label{lemma2}
For every $\epsilon\in (0,\epsilon_0)$, the bound
\[
\left| \frac{\Re \widetilde F \ast \cal F_n}{\Re \widetilde F}-1
\right|\lesssim \epsilon
\]
holds if $\theta\in [-\pi,\pi]$ and
$n>n_{(\epsilon)}$.
\end{lemma}
\begin{proof} Notice first that $\Re \widetilde F\to \Re (1-e^{i\theta})^{-\epsilon}$ and $\Re \widetilde F \ast \cal F_n\to \Re (1-e^{i\theta})^{-\epsilon}$ uniformly over $\theta: \pi/2\le|\theta|\le \pi$. Thus, we only need to consider $\theta:|\theta|\le \pi/2$.

Using $|(\Re \widetilde F) \ast \cal F_n-\Re \widetilde F|=|\Re( \widetilde F \ast \cal F_n- \widetilde F)|\le|\widetilde F \ast \cal F_n- \widetilde F|$, we can rewrite
\[
\left|\frac{\Re \widetilde F \ast \cal F_n}{\Re \widetilde F}-1\right|=\left|\frac{\Re \widetilde F \ast \cal F_n-\Re \widetilde F}{\Re \widetilde F}\right|\lesssim \frac{|h_n(\theta)|^2}{\Re h_n(\theta)}\cdot \left|\int_{-\pi}^\pi \frac{\sin^2(nx/2)}{n\sin^2(x/2)}  \Bigl(h_n^{-1}(\theta-x)-h_n^{-1}(\theta) \Bigr)               dx\right|\,.
\]
In the previous Lemma, we showed that $|h_n|\sim \Re h_n$, thus
\[
\left|\frac{\Re \widetilde F \ast \cal F_n}{\Re \widetilde F}-1\right|\lesssim  \left|\int_{-\pi}^\pi \frac{\sin^2(nx/2)}{n\sin^2(x/2)}  \frac{h_n(\theta-x)-h_n(\theta)}{h_n(\theta-x)}               dx\right|\,.
\]
Now, we can substitute \eqref{arom3} and \eqref{blue1} into this formula.
\[
\left|\frac{\Re \widetilde F \ast \cal F_n}{\Re \widetilde F}-1\right|\lesssim
  \int_{-\pi}^\pi \frac{\sin^2(nx/2)}{nx^2} \left| \frac{M_n(n(\theta-x))-M_n(n\theta)}{M_n(n(\theta-x))}       \right|        dx
\]
\[
+\epsilon \int_{-\pi}^\pi \frac{\sin^2(nx/2)}{nx^2} \frac{|M_n(n(\theta-x))|+|M_n(n\theta)|}{M_n(n(\theta-x))}              dx+n^{-1} \int_{-\pi}^\pi\left| \frac{h_n(\theta-x)-h_n(\theta)}{h_n(\theta-x)}    \right|           dx\,.
\]
We apply Lemma \ref{lemma1} to bound the last term by
\[
Cn^{-1} \int_{-\pi}^\pi \frac{1}{|\theta-x|^{\epsilon}}              dx\lesssim n^{-1}\,,
\]
which is smaller than $\epsilon$ if $n>n_{(\epsilon)}$.

Denote $\widehat x:=nx, \widehat \theta:=n\theta$. We need to bound
\[
 \int_{-\pi n}^{\pi n} \frac{\sin^2(\widehat x/2)}{\widehat x^2} \left| \frac{M_n(\widehat \theta-\widehat x)-M_n(\widehat \theta)}{M_n(\widehat \theta-\widehat x)}       \right|        d\widehat x+ \epsilon\int_{-\pi n}^{\pi n} \frac{\sin^2(\widehat x/2)}{\widehat x^2}  \frac{|M_n(\widehat \theta-\widehat x)|+|M_n(\widehat \theta)|}{|M_n(\widehat \theta-\widehat x)|}               d\widehat x\,.
\]
We use the estimate for $M_n(x)\sim 1+|x|^\epsilon$ to estimate the second term by
\begin{equation}\label{blue3}
\epsilon\int_{-\pi n}^{\pi n} \frac{\sin^2(\widehat x/2)}{\widehat x^2} \cdot \frac{1+|\widehat \theta-\widehat x|^\epsilon+|\widehat \theta|^\epsilon}{1+|\widehat \theta-\widehat x|^\epsilon}               d\widehat x\lesssim \epsilon\,,
\end{equation}
because $|\widehat \theta|\le |\widehat \theta-\widehat x|+|\widehat x|$ and so $|\widehat \theta|^\epsilon\le |\widehat x|^\epsilon+|\widehat x-\widehat\theta|^\epsilon$.

Consider the first term now. It is equal to 
\begin{equation}\label{blue4}
 \int_{|\widehat x|<\log n} \frac{\sin^2(\widehat x/2)}{\widehat x^2} \left| \frac{M_n(\widehat \theta-\widehat x)-M_n(\widehat \theta)}{M_n(\widehat \theta-\widehat x)}       \right|        d\widehat x+\int_{\pi n>|\widehat x|>\log n} \frac{\sin^2(\widehat x/2)}{\widehat x^2} \left| \frac{M_n(\widehat \theta-\widehat x)-M_n(\widehat \theta)}{M_n(\widehat \theta-\widehat x)}       \right|        d\widehat x\,.
\end{equation}
If we argue as we did in \eqref{blue3}, the last integral is bounded by
\[
C\int_{|\widehat x|>\log n}\frac{1}{|\widehat x|^{2-\epsilon}}d\widehat x=O(\log^{-1+\epsilon} n)<\epsilon
\]
for $n$ sufficiently large.

For the first integral in \eqref{blue4}, we recall that $\widehat \theta: |\widehat\theta|<\pi n/2$ and use an estimate \eqref{blue2} for the derivative of $M_n$. Keeping in mind $|\widehat x|<\log n$, we obtain
\[
|M_n(\widehat \theta-\widehat x)-M_n(\widehat \theta)|=\left|\int_{\widehat \theta}^{\widehat \theta-\widehat x} M_n'(\xi)d\xi\right|\lesssim
n^{\epsilon-2} |\widehat x|+\left|\int_{\widehat \theta}^{\widehat \theta-\widehat x}\epsilon (1+|\xi|)^{\epsilon-1}d\xi\right|
\lesssim  n^{\epsilon-2}\log n+
\epsilon (1+|\widehat x|^{1/2})
\]
by Cauchy-Schwarz, provided that $\epsilon_0<\frac{4}{10}$. Thus, taking into account the lower bound $M_n\gtrsim 1$, we have an estimate
\[
 \int_{|\widehat x|<\log n} \frac{\sin^2(\widehat x/2)}{\widehat x^2} \left| \frac{M_n(\widehat \theta-\widehat x)-M_n(\widehat \theta)}{M_n(\widehat \theta-\widehat x)}       \right|        d\widehat x\lesssim  n^{\epsilon-2}\log n+\epsilon \int_{|\widehat x|<\log n} \frac{\sin^2(\widehat x/2)}{\widehat x^2} (1+|\widehat x|^{1/2})      d\widehat x\lesssim \epsilon\,,
\]
provided that $\epsilon_0\in (0,\frac{4}{10})$ and $n>n_{(\epsilon)}$.

\end{proof}
Finally, we mention that the choice of $\epsilon_0$ is made to have $1+O(\epsilon)\in (\frac 12,2)$ in, e.g., \eqref{tockta}. We also needed $\epsilon_0<\frac{4}{10}$.

\bigskip

\section{Appendix D: The properties of auxiliary polynomials, II}

In this Appendix, we will study the polynomials introduced in subsection 3.2. Recall that
\begin{equation}\label{grey101}
 H_n= 2(1-e^{i\theta})^{\alpha}\ast \cal{K}_{[n/2]}\,,
\end{equation}
where $\cal{K}_l$ is Jackson kernel and $\alpha\in (\frac 12,1)$. Parameter $\tau$ was chosen as $\tau=1-\alpha$.

 We again start with estimates for $(1-e^{i\theta})^\alpha$. We have
\begin{equation}\label{outside}
\Re (1-e^{i\theta})^\alpha\gtrsim  \tau |(1-e^{i\theta})|^\alpha, \quad    \Re (1-e^{i\theta})^\alpha\gtrsim  \tau |\Im (1-e^{i\theta})^\alpha|,\quad        |
(1-e^{i\theta})^\alpha|\sim |\theta|^{\alpha}\,,
\end{equation}
if $\theta\in [-\pi,\pi]$ and $\alpha\in (\frac 12,1)$. The function $\Im (1-e^{i\theta})^\alpha$ is odd in $\theta$ and the function $\Re (1-e^{i\theta})^\alpha$ is even.
We will also need asymptotical expansion around the origin. We write
\[
(1-e^{i\theta})^\alpha=(2\sin(\theta/2))^{\alpha}(\sin(\theta/2)-i\cos(\theta/2))^\alpha=(2\sin(\theta/2))^{\alpha}e^{-i\alpha \nu}\,,
\]
where
$
\nu:=\arctan \cot (\theta/2)
$. Therefore, we  get
\begin{equation}\label{lomk}
\Re (1-e^{i\theta})^\alpha=|2\sin(\theta/2)|^{\alpha}\cos(\alpha \arctan \cot (\theta/2)),
\end{equation}
\begin{equation}\label{lomk1}
\Im (1-e^{i\theta})^\alpha=-|2\sin(\theta/2)|^{\alpha}\sin(\alpha \arctan \cot (\theta/2)),
\end{equation}
When $\theta\to 0$, 
\[
\cot (\theta/2)=2\theta^{-1}(1+O(\theta^2))\,.
\]
When $t\to\infty$,
\[
\arctan t={\rm sgn} (t)\cdot \Bigl(\pi/2+O(|t|^{-1})\Bigr)\,.
\]
Therefore, expanding $\sin(\theta/2)$ at the origin, we have
\begin{equation}\label{redutto}
\Re (1-e^{i\theta})^\alpha=C|\theta|^\alpha(\cos(\alpha\pi/2)+O(\theta)), \,\Im (1-e^{i\theta})^\alpha=-C {\rm sgn } (\theta)\cdot|\theta|^\alpha(\sin(\alpha\pi/2)+O(\theta))\,.
\end{equation}

 The standard property of the convolution yields 
\begin{equation}\label{nordf}
H_n(\theta)\to 2(1-e^{i\theta})^\alpha, \quad n\to\infty
\end{equation}
uniformly over $\theta\in [-\pi,\pi]$. Therefore, for every $\delta>0$, there is $n_{(\delta)}\in \mathbb{N}$, such that 
\begin{equation}\label{outside1}
1\gtrsim \Re H_n\gtrsim \tau |\theta|^\alpha, \quad |H_n|\sim |\theta|^{\alpha}
\end{equation}
if $n>n_{(\delta)}$ and $|\theta|\in [\delta,\pi]$. This follows from \eqref{outside} and uniform convergence \eqref{nordf}.

\begin{lemma}\label{grey4}

 The function $\Im H_n(\theta)$ is odd in $\theta$.
There exists $\tau_0>0$ such that for every $\tau\in (0,\tau_0)$ there is $n_{(\tau)}\in \mathbb{N}$ so that $H_n$ satisfies the following properties for all $n>n_{(\tau)}$. 
\begin{itemize}
\item For the real part,
\begin{equation}\label{grey2}
  \Re H_n\sim  \tau(n^{-\alpha}+|\theta|^\alpha), \quad{\rm if}\, |\theta|<\tau^2\,.
\end{equation}
\item For the absolute value,
\begin{equation}\label{grey3}
|H_n|\sim 
\left\{
\begin{array}{cc}
n^\tau|\theta|+\tau n^{-\alpha}, & |\theta|<1/n,\\
|\theta|^\alpha,&  |\theta|>1/n\,.
\end{array}
\right.
\end{equation}

\item  For the argument of $H_n$, 
 \begin{equation}\label{orange1}
-\pi/2+C\tau< \arg H_n<\pi/2-C\tau, \,{\rm if}\,  \theta\in [-\pi,\pi)\,.
 \end{equation}
 \end{itemize}
\end{lemma}

\begin{proof}
The fact that $\Im H_n$ is odd is immediate because $\Im (1-e^{i\theta})^\alpha$ is odd.

To see \eqref{orange1}, we notice that the Jackson kernel  and $\Re (1-e^{i\theta})^\alpha$ are both nonnegative, so the second estimate in \eqref{outside} yields
\[
|\Im H_n|\lesssim \tau^{-1}\Re H_n\,.
\]
We then have
\[
|\tan \arg H_n|\lesssim \tau^{-1}\,,
\]
which gives \eqref{orange1}.

Fix $\delta>0$ and consider $|\theta|<\delta$. Denote $l:=[n/2]$ and notice that $l\sim n$.
The estimates on the Jackson kernel imply
\[
\Bigl( (1-z)^{\alpha}\ast
\cal{K}_l\Bigr)(\theta)=\int_{|x|<\delta}
(1-e^{i(\theta-x)})^{\alpha}\cal{K}_l(e^{ix})dx+O(\delta^{-3}n^{-3})\,.
\]
If $|x|<\delta$ and $|\theta|<\delta$, then (check  \eqref{lomk},\eqref{lomk1},\eqref{redutto}) we have
\[
\cos \left(\alpha\arctan \cot ((\theta-x)/2)\right)=\cos
(\alpha\pi/2)+O(\delta),
\]
\[ \quad \sin \left(\alpha\arctan \cot
((\theta-x)/2)\right)={\rm sgn}(\theta-x)\cdot \sin
(\alpha\pi/2)+O(\delta)\,.
\]
Therefore,
\[
\Re \Bigl( (1-z)^{\alpha}\ast \cal{K}_l\Bigr)(\theta)=(\cos
(\alpha\pi/2)+O(\delta))\int_{|x|<\delta}|2\sin((\theta-x)/2)|^{\alpha}\cdot
\cal{K}_l(e^{ix})dx+O(\delta^{-3}n^{-3})
\]
and
\begin{eqnarray}
\Im \Bigl( (1-z)^{\alpha}\ast \cal{K}_l\Bigr)(\theta)= -\sin
(\alpha\pi/2)\int_{|x|<\delta}|2\sin((\theta-x)/2)|^{\alpha}\cdot {\rm
sgn}(\theta-x) \cdot \cal{K}_l(e^{ix})dx+  \nonumber \\
O(\delta)\int_{|x|<\delta}|2\sin((\theta-x)/2)|^{\alpha}\cdot
\cal{K}_l(e^{ix})dx+O(\delta^{-3}n^{-3})\,.\hspace{2cm}\label{kuklya}
\end{eqnarray}
Take $\delta=\tau^2$. Then, we choose $\tau_0$ such that 
\begin{equation}\label{perg1}
\cos (\alpha\pi/2)+O(\delta)\sim \tau\,.
\end{equation}
for all $\tau\in (0,\tau_0)$.
Then,
\[
\int_{|x|<\tau^2}|2\sin((\theta-x)/2)|^{\alpha}\cdot
\cal{K}_l(e^{ix})dx\sim
\int_{|x|<\tau^2}|\theta-x|^{\alpha}\cdot
\frac{\sin^4(lx/4)}{l^3x^4}dx\,.
\]
As in the previous Appendix, we introduce $\widehat x:=lx, \widehat \theta:=l\theta$.
The last integral becomes
\begin{equation}\label{glyur}
l^{-\alpha}\int_{|\widehat x|<l\tau^2}|\widehat \theta-\widehat x|^{\alpha}\cdot
\frac{\sin^4(\widehat x/4)}{\widehat x^4}d\widehat x\sim l^{-\alpha}(1+|\widehat \theta|^\alpha)\,,
\end{equation}
if $|\widehat \theta|<l\tau^2$.  Combining the estimates and recalling that $l\sim n$, we get
\[
\Re \Bigl( (1-e^{i\theta})^{\alpha}\ast \cal{K}_l\Bigr)(\theta)\sim \tau^{-6}n^{-3}+\tau(n^{-\alpha}+|\theta|^\alpha)\,.
\]
This proves \eqref{grey2}. For the imaginary part, \eqref{kuklya} gives
\[
\Im \Bigl( (1-z)^{\alpha}\ast \cal{K}_l\Bigr)(\theta)=-\sin(\alpha \pi/2) J_1+J_2+O(\delta^{-3}n^{-3})\,,
\]
\[
 J_1:=\int_{|x|<\delta}|2\sin((\theta-x)/2)|^{\alpha}\cdot {\rm
sgn}(\theta-x) \cdot \cal{K}_l(x)dx\,,
\]
\[
J_2:=O(\delta)\int_{|x|<\delta}|2\sin((\theta-x)/2)|^{\alpha}\cdot
\cal{K}_l(e^{ix})dx\,.
\]
For $J_2$, 
\begin{equation}\label{posya}
J_2= O(\tau^3)(n^{-\alpha}+|\theta|^\alpha)
\end{equation}
due to \eqref{glyur} and the choice of $\delta$.

Consider $J_1$.   We  write Taylor expansion for $\sin$
\[
\sin(\theta-x)=\theta-x+O((\theta-x)^2)\,.
\]
For the Jackson kernel, we have the asymptotical representation
\[
\cal{K}_l(\theta)=\frac{C}{l^3}\frac{\sin^4(l\theta/2)}{\theta^4}+O(l^{-1})\,.
\]
Substituting these two expressions into the formula for $J_1$, we get
\[
J_1=J_{(1,1)}+\epsilon_{(n,\theta)}\,.
\]
The expression for the main term, $J_{(1,1)}$, is
\[
J_{(1,1)}=Cl^{-3}\int_{|x|<\delta}|\theta-x|^{\alpha}\cdot {\rm
sgn}(\theta-x) \cdot \frac{\sin^4(lx/2)}{x^4}dx\,.
\]
Carrying out the standard bounds, we get
\begin{equation}\label{dusya}
|\epsilon_{(n,\theta)}|\lesssim  n^{-1}
\end{equation}
if $n>n_{(\tau)}$.
In $J_{(1,1)}$, perform the change of variables  $\widehat x=lx, \widehat\theta=l\theta$ to get
\[
J_{(1,1)}=Cl^{-\alpha}S(\widehat \theta), \,S(\widehat\theta):=\int_{|\widehat x|<l\tau^2} {\rm sgn} (\widehat\theta-\widehat x)\cdot  |\widehat \theta-\widehat x|^{\alpha}\cdot
\frac{\sin^4(\widehat x/4)}{\widehat x^4}d\widehat x\,.
\]
Notice that $S(0)=0$ and  
\[
S'(\widehat \theta)=\alpha \int_{|\widehat x|<l\tau^2} |\widehat x-\widehat \theta|^{\alpha-1}\frac{\sin^4(\widehat x/4)}{\widehat x^4}d\widehat x\sim 1+|\widehat\theta|^{\alpha-1}\,.
\]
Since $S(\widehat\theta)=S(0)+\int_0^{\widehat \theta} S'(\xi)d\xi$, we have
\[
|S(\widehat \theta)|\sim
\left\{
\begin{array}{cc}
|\widehat \theta|, &|\widehat \theta|<1\,,\\
|\widehat \theta|^{\alpha}, & |\widehat \theta|>1\,.
\end{array}
\right.
\]
Going back to the original variable, we get
\[
|J_{(1,1)}(\theta)|\sim
\left\{
\begin{array}{cc}
n^\tau|\theta|, &|\theta |<1/n\,,\\
|\theta|^\alpha, & 1/n<|\theta|\,.
\end{array}
\right.
\]
Now, we can combine these bounds to control $|H_n|$. Notice that $|H_n|\sim |\Re H_n|+|\Im H_n|$.
For $|\theta|>1/n$, $J_{(1,1)}$ gives the main contribution and $|H_n|\sim |\theta|^\alpha$. When $\theta\in [-1/n,1/n]$, we have $\Re H_n\sim \tau n^{-\alpha}$. Then,
$|\epsilon_{(n,\theta)}|$ is negligible compared to $\Re H_n$, as can be seen from \eqref{dusya}. Thus,
\[
|H_n|\sim n^\tau|\theta|+\tau n^{-\alpha}\,.
\]
Note carefully that  $n^\tau|\theta|\sim \tau n^{-\alpha}$ when $|\theta|\sim \tau/n$. 

Thus, we have \eqref{grey3} for $|\theta|<\tau^2$. To prove \eqref{grey3} for all $\theta$, it is sufficient to notice that outside the interval $[-\tau^2,\tau^2]$, \eqref{nordf} gives \eqref{grey3} immediately due to the properties of the limiting function $(1-e^{i\theta})^\alpha$.

In conclusion, we want to mention that $\tau_0$ has been chosen ``small enough'' to make the asymptotical calculation \eqref{perg1} valid.
\end{proof}

Consider $Q_n$ given by \eqref{grey6}, i.e.,
\begin{equation}\label{grey7}
Q_n=(1-e^{i\theta})^{-\alpha/2}\ast \cal{F}_n\,.
\end{equation}
The following Lemma can be proved in the standard way.
\begin{lemma}\label{blue77}
If $\alpha\in (1/2,1)$, then there is $n_0\in \mathbb{N}$ so that 
\[
\Re Q_n\sim \left\{
\begin{array}{cc}
n^{\alpha/2}, &\, |\theta|<n^{-1},\\
|\theta|^{-\alpha/2}, &\, n^{-1}<|\theta|
\end{array}
\right.
\]
and
\[
|Q_n|\sim \left\{
\begin{array}{cc}
n^{\alpha/2},& \, |\theta|<n^{-1},\\
|\theta|^{-\alpha/2}, &\, n^{-1}<|\theta|
\end{array}
\right.
\]
for all $n>n_0$.
\end{lemma}

\begin{proof} Let $\gamma:=\alpha/2$. We have $\gamma\in (\frac 14,\frac 12)$.
The proof is similar to that of Lemma 1, except it is easier since we are not concerned with how the estimates depend on small parameter. We start by writing
\begin{equation}\label{arom5}
\Re (1-e^{i\theta})^{-\gamma}\sim |\theta|^{-\gamma}, \, |\Im (1-e^{i\theta})^{-\gamma}|\lesssim |\theta|^{-\gamma}\,.
\end{equation}
The last inequality implies
\[
|\Im Q_n|\lesssim \Re Q_n\,.
\]
Therefore, we only need to focus on $\Re Q_n$. Substitute the formula \eqref{arom3} into the convolution and use \eqref{arom5} to get
\[
\Re Q_n(\theta)\sim n^{-1}\int_{-\pi}^\pi |\theta-x|^{-\gamma}  \frac{\sin^2(nx)}{x^2}dx+O(n^{-1})\,.
\]
Consider the integral and change variables $\widehat \theta:=n\theta, \widehat x:=nx$. This gives
\[
n^\gamma \int_{-\pi n}^{\pi n} {|\widehat x-\widehat \theta|^{-\gamma}}\frac{\sin^2(\widehat x)}{\widehat x^2}d\widehat x\sim \frac{n^\gamma}{1+|\widehat \theta|^\gamma}\sim
\left\{
\begin{array}{cc}
n^\gamma,& \quad |\theta|<1/n,\\
|\theta|^{-\gamma}, & |\theta|>1/n
\end{array}\,.
\right.
\]
Comparing the last quantity to $O(n^{-1})$, we  finish the proof.
\end{proof}

{\Large \part*{Acknowledgement.}} The work done in the last section
 of the paper was supported by RSF-14-21-00025 and author's research
on the rest of the paper was supported by the grant NSF-DMS-1464479.

\end{document}